\documentclass[12pt]{amsart}
\usepackage{latexsym,amssymb,amsmath}
\usepackage{amsmath,amsfonts}
\usepackage{epsfig}
\usepackage{graphicx}
\usepackage{verbatim,xcolor}

\newcommand{\norm}[1]{\left\lVert#1\right\rVert}
\theoremstyle{plain}
\newtheorem{theorem}{Theorem}[section]

\newtheorem{lemma}[theorem]{Lemma}
\newtheorem{Definition}[theorem]{Definition}
\newtheorem{proposition}[theorem]{Proposition}

\theoremstyle{remark}

\allowdisplaybreaks

\newcommand{\Hmm}[1]{\leavevmode{\marginpar{\tiny%
$\hbox to 0mm{\hspace*{-0.5mm}$\leftarrow$\hss}%
\vcenter{\vrule depth 0.1mm height 0.1mm width \the\marginparwidth}%
\hbox to 0mm{\hss$\rightarrow$\hspace*{-0.5mm}}$\\\relax\raggedright #1}}}

\begin{document}

\title[SMOOTHING FOR THE DAVEY-STEWARTSON SYSTEM]{GLOBAL SMOOTHING FOR THE DAVEY-STEWARTSON\\ SYSTEM ON $\mathbb{R}^2$}

\author[E. Ba\c{s}ako\u{g}lu] {Eng\.{I}n Ba\c{s}ako\u{g}lu}
\thanks{The author was partially supported by the T\"UB\.ITAK grant 118F152 and the Bo\u gazi\c ci University Research Fund grant BAP-14081.} 
\email{engin.basakoglu@boun.edu.tr}

\address{Department of Mathematics,
Bo\u gazi\c ci University, 
Bebek 34342, Istanbul, Turkey}

\begin{abstract}
In this paper we study the regularity properties of solutions to the Davey-Stewartson system. It is shown that for initial data in a Sobolev space, the nonlinear part of the solution flow resides in a smoother space than the initial data for all times. We also obtain that the Sobolev norm of the nonlinear part of the evolution grows at most polynomially. As an application of the smoothing estimate, we study the long term dynamics of the forced and weakly damped Davey-Stewartson system. In this regard we give a new proof for the existence and smoothness of the global attractors in the energy space.     

\end{abstract}

\maketitle

\section{INTRODUCTION}\label{ıntro}
The Davey-Stewartson (DS) equations in dimensionless are given by couple of equations of the form 
\begin{align} \label{eqn}
\begin{cases}
i\partial_t u+c_0 \partial_{x}^2u +\partial_{y}^2u=c_1|u|^2u+c_2u\partial_x\phi \\ \partial_x^2\phi+c_3\partial_{y}^2\phi=\partial_x(|u|^2) \end{cases}
\end{align}
where $u=u(x,y,t)$ is complex valued and $\phi=\phi(x,y,t)$ is real valued functions that represent amplitude and mean velocity potential respectively, here the constants are real numbers and their signs determine the character of the equation. The systems \eqref{eqn} were first derived by Davey and Stewartson \cite{Davey}, Benney and Roskes \cite{Benney}, Djordjevic
and Redekopp \cite{Djordjevic} and model the time evolution of $2$D surface of water waves that propagate predominantly in one direction but the wave amplitude is modulated slowly in the horizontal directions. In \cite{Djordjevic}, Djordjevic
and Redekopp showed that the parameter $c_3$ can be negative when capillary effects are important. According to the signs of $c_0$ and $c_3$ respectively, the system \eqref{eqn} is classified as follows \begin{center}
\begin{tabular}{ c c c c }
 Elliptic-Elliptic & Elliptic-Hyperbolic & Hyperbolic-Elliptic & Hyperbolic-Hyperbolic\\ 
 $(+,+)$ & $(+,-)$ & $(-,+)$ & $(-,-)$ 
     
\end{tabular}
\end{center}
  DS systems are very well studied in terms of well-posedness and stability, blow-up profiles, existence of standing and travelling waves. The investigation of the system \eqref{eqn} in terms of well-posedness was initiated by Ghidaglia and Saut \cite{Ghidaglia} who established the local well-posedness in the elliptic-elliptic, elliptic-hyperbolic and hyperbolic-elliptic cases. More precisely, they studied the local and global properties of the elliptic-elliptic and the hyperbolic-elliptic cases in $L^2,H^1,H^2$ also in the elliptic-hyperbolic case, they obtained a global existence of weak solution of \eqref{eqn} under smallness assumption for data in $L^2$. Linares and Ponce \cite{Linares} showed that under some smallness assumptions on the data, elliptic-hyperbolic and hyperbolic-hyperbolic cases of \eqref{eqn} are locally well-posed in the spaces $H^s(\mathbb{R}^2)\cap H^6(\mathbb{R}^2:r^6dxdy)$ for $s\geq 12$, and $H^s(\mathbb{R}^2)\cap H^3(\mathbb{R}^2:r^2dxdy)$ for $s\geq 6$ respectively. As regards to the initial value problem posed on the  $2$-torus Godet \cite{Godet} obtained a local well-posedness result for the hyperbolic-elliptic problem in $H^s(\mathbb{T}^2)$ for $s>1/2$ as well as a blow-up rate for this equation. Concerning the half-plane problem, Fokas \cite{Fokas} studied the DS equation on the half-plane by using the inverse scatering transform techniques along with the formulation of a $d$-bar problem for a sectionally non-analytic function. As for the problem of global well-posedness, it is conjectured that the elliptic-elliptic type of \eqref{eqn} is globally well-posed in $H^s$ for all $s\geq0$. Towards this conjecture, Shen and Guo \cite{Guo} proved that the initial value problem of \eqref{eqn} in the elliptic-elliptic case (with some assumptions on the constants) is globally well-posed for data in $H^s(\mathbb{R}^2)$, for $s>4/7$.  Thereafter, Yang et al. \cite{Yang} improved this result by establishing global well-posedness in $H^s(\mathbb{R}^2)$ for $s>2/5$ where they took advantage of the $I$-method. In particular they obtained a polynomial in time bound for the $H^s$ norm of the solution for $s>2/5$. Some of the other results regarding the system \eqref{eqn} can be found in \cite{2,3,4,5,6}.   
Upon considering elliptic-elliptic type of the system \eqref{eqn} we will study the initial value problem
\begin{align}\label{ee}
\begin{cases}
i\partial_t u+ \Delta u=c_1|u|^2u+c_2u\partial_x\phi,\ \  (x,y)\in \mathbb{R}^2,\ t\in \mathbb{R}, \\ \partial_x^2\phi+ \partial_{y}^2\phi=\partial_x(|u|^2)\\ u(x,y,0)=u_0(x,y) \in H^s(\mathbb{R}^2). 
\end{cases}
\end{align}
To reformulate \eqref{ee} in a better form we implement the Fourier transform in the spatial variable to the second equation of \eqref{ee} so that the system reduces to a single equation \begin{align*} 
i\partial_t u +\Delta u=c_1|u|^2u+ c_2 K(|u|^2)u 
\end{align*} where $K$ is the pseudo-differential operator with symbol $\alpha$ given by \begin{align*} \widehat{K(f)}(\xi)=\alpha(\xi)\Hat{f}(\xi)\  \text{and}\ \alpha(\xi)=\frac{\xi_1^2}{|\xi|^2},\ \xi=(\xi_1,\xi_2)\neq (0,0).
\end{align*}
Therefore by the Duhamel formula, the equation \eqref{ee} is equivalent to  \begin{align}\label{duhamel} 
u(x,y,t)=e^{it \Delta }u_{0}-i \int_0^t e^{i(t-\tau)\Delta}(c_1|u|^2u+c_2K(|u|^2)u)(\tau)d\tau 
\end{align}
where $e^{it\Delta}$ denotes the free solution operator of the corresponding linear problem for the equation \eqref{ee}. In this paper our primary goal is to obtain smoothing properties of \eqref{ee} globally in time, therefore in order to take advantage of global well-posedness result of Theorem \ref{global} below, we will assume the restriction that $c_1+c_2>0$. In the absence of this restriction, smoothing argument here works in the local sense only. The smoothing in this text means that the nonlinear part of the solution flow in relation to \eqref{duhamel} lies in a more regular space than the initial data belongs to. To make it rigorous, below we state our result:
\begin{theorem}\label{mainthm}
Let $c_1+c_2>0$ . Fix $s>\frac{1}{2}$ and $a<\min(\frac{1}{2},s-\frac{1}{2})$. Consider the solution to IVP \eqref{ee} on $\mathbb{R}^2\times \mathbb{R}$ with data $u_0\in H^s(\mathbb{R}^2)$. Suppose that there is an a priori growth bound $\norm{u(t)}_{H^s}\leq C(\norm{u_0}_{H^s})\langle t \rangle^{\beta(s)}$ for some $\beta(s)$. Then \begin{align}\label{smoothing}
u(x,y,t)-e^{it\Delta}u_0 \in C_t^0H^{s+a}_{x,y}  
\end{align}furthermore, we have the growth bound \begin{align*}
\norm{u(t)-e^{it\Delta }u_0}_{H^{s+a}}\leq C(s,a,\norm{u_0}_{H^s})\langle t\rangle^{1+\beta(s)(3+\frac{2}{s})}.
\end{align*}
\end{theorem} Let $I$ denote the identity operator and $K$ denote the multiplier operator introduced as above, hence to be able to prove the Theorem \ref{mainthm} we will need the key trilinear estimate:
\begin{proposition}\label{key}
For $s>\frac{1}{2}$, $a<min(\frac{1}{2}, s-\frac{1}{2})$ and $b=\frac{1}{2}+\epsilon$ for $\epsilon>0$ sufficiently small, we have
\begin{align}\label{keyest}
||(c_1I+c_2K)(u\overline{v})w||_{X^{s+a,b-1}} \lesssim ||u||_{X^{s,b}}||v||_{X^{s,b}}||w||_{X^{s,b}}.
\end{align}  
\end{proposition}
Note that  we can replace the $X^{s,b}$ norm in the above proposition with the time restricted version of this norm. 

In the remainder part, we present an application of the smoothing estimate--an alternative proof of the existence of a global attractor for the forced and damped Davey-Stewartson system in the energy space $H^1$. In \cite{Wang}, Wang and Guo obtained the existence of a global attractor in $H^1$ yet with a proof based on the splitting argument that requires more regular initial data to reach the compactness. The smoothing effect replaces the splitting method of \cite{Wang}, as a result simplifying the proof and also as a byproduct, gives us that the global attractor is indeed a compact subset of $H^{\frac{3}{2}-}$. Thus we consider the forced and damped Davey-Stewartson system
\begin{align}\label{elelforced}
\begin{cases}
iu_t +\Delta u+i\delta u=c_1|u|^2u+c_2u\partial_x \phi +f,\,\,(x,y)\in\mathbb{R}^2  \\ \Delta \phi=\partial_x(|u|^2)
\end{cases}
\end{align}
where the forcing term $f\in L^2(\mathbb{R}^2)$ is time independent, $\delta>0$, $c_1\geq 0$ and $c_1+c_2\geq 0$. The use of Theorem \ref{attractortemam} below which is provided by a smoothing estimate for the dissipative problem \eqref{elelforced} yield the following result: 

\begin{theorem}\label{thm:attractor}
Consider the forced and weakly damped Davey-Stewartson system \eqref{elelforced} on $\mathbb{R}^2\times[0,\infty)$ with the initial data $u(x,0)=u_0(x)\in H^1(\mathbb{R}^2)$. Then the equation \eqref{elelforced} possesses a global attractor in $H^1(\mathbb{R}^2)$. Furthermore, for any $a\in(0,\frac{1}{2})$ the global attractor is a compact subset of $H^{1+a}(\mathbb{R}^2)$.
\end{theorem}
We now briefly explain the organization of the paper. In section \ref{prelim}, we introduce the notations, function spaces and necessary tools for the proof of the Theorem \ref{mainthm}. Also as regards to a global attractor of the equation \eqref{elelforced}, we give necessary definitions and a compactness theorem. In section \ref{method}, we discuss Tao's $[k;Z]$ multiplier method in order to prove the key trilinear estimate \eqref{keyest}. The proof of \eqref{keyest} will be given in Section \ref{trilinearestimate}. In Section \ref{pfmainthm}, we prove the Theorem \ref{mainthm} and finally Section \ref{attractor} is devoted to prove the existence of a global attractor for the Davey-Stewartson system. 
\section{Notation \& Preliminaries}\label{prelim}
Throughout the text we use $A\lesssim B$ to denote an estimate $A\leq CB$ for some absolute constant C. We write $A\sim B$ if $A\lesssim B$ and $B\lesssim A$. We also write $A\ll B$ to denote an estimate $A\leq \frac{1}{C}B$ for some sufficiently large constant $C$. Moreover we define $\langle\cdot\rangle=\sqrt{1+|\cdot|^2}$. For $s,b\in\mathbb{R}$, we require $X^{s,b}$ spaces corresponding to the evolution $u$ of the DS system that is defined by means of the norm $$\norm{u}_{X^{s,b}}=\norm{\langle \xi\rangle^s\langle\tau+|\xi|^2 \rangle^b\widehat{u}(\xi,\tau)}_{L^2_{\xi,\tau}}.$$ Localized $X^{s,b}$ space is also defined by  $$\norm{u}_{X^{s,b}_{\delta}}=\inf_{\tilde{u}=u,\, |t|\leq\delta}\norm{\tilde{u}}_{X^{s,b}}.$$ For a general discussion, let $\varphi$ be the smooth function satisfying $\varphi(t)=0$ if $|t|\geq 2$, and $\varphi(t)=1$ for $|t|\leq 1$. Also let $\varphi_{\delta}(t)=\varphi(t/\delta)$ for $0<\delta \leq 1$. Consider the Cauchy problem  \begin{align} \label{Cauchy}
\begin{cases}
iu_t + h(\frac{1}{i}\nabla)=F(u)\\ u(0)=u_0                                                           
\end{cases}
\end{align}
where $h$ is a real polynomial on $\mathbb{R}^d$ and $F$ is a non-linear function. Recall that in order to solve the equation \eqref{Cauchy} locally in time, say, in the interval $[-\delta,\delta]$, we write an integral equation corresponding to \eqref{Cauchy} as \begin{align*}u(t)&=\varphi(t)W(t)u_0+i\varphi_{\delta}(t)\int_0^tW(t-s)F(u(s))ds \\& =: \varphi(t)W(t)u_0+i\varphi_{\delta}(t)[W*_tF(u)](t)
\end{align*}
where $W(t)=exp(ith(\frac{1}{i}\nabla))$ is the group of unitary operators that solves the corresponding linear equation and $*_t$ denotes the convolution with respect to the time variable restricted to the interval $[0,t]$. With these notations we have the following lemma which will be important in attaining Theorem \ref{mainthm}. 

\begin{lemma}[See \cite{Ginibre}.]\label{delta}
Let $-\frac{1}{2}<b'\leq 0 \leq b \leq b'+1 $. Then $$\norm{ \varphi_{\delta}(t)W*_tF(u)(t)}_{X^{s,b}} \lesssim \delta^{1-b+b'}\norm{F}_{X^{s,b'}_\delta}.$$
\end{lemma}

Consider the local problem \eqref{ee} with the local existence time $\delta$. We will quantify the dependence of $\delta$ to an initial data which we use in the proof of Theorem \ref{mainthm}. From the dilation symmetry of this equation, assuming that $(u,\phi)$ solve \eqref{ee} with initial data $u_0$ on $[0,\lambda^{-2}]$, we come up with the symmetry solutions $$u^{\lambda}(x,y,t)=\lambda^{-1}u(x/\lambda,y/\lambda,t/\lambda^2),\,\,\phi^{\lambda}(x,y,t)=\lambda^{-1}\phi(x/\lambda,y/\lambda,t/\lambda^2)$$ with data $u^{\lambda}_0(x,y)=\lambda^{-1}u_0(x/\lambda,y/\lambda)$ which solve the equation on $[0,1]$. Thus for $\lambda>1$, by comparing the $H^s$ norms of $u_0$ and $u_0^{\lambda}$,  the solution $(u,\phi)$ can be defined with respect to the local existence time
\begin{align}\label{existtime}
\delta\sim (C+\norm{u_0}_{H^s})^{-\frac{2}{s}}
\end{align} where $C=C(\norm{u_0}_{L^2})$.

Next we discuss a global wellposedness result for the elliptic-elliptic problem.  We will exploit it once iterating our local result. Note that recalling the Sobolev index range $s>1/2$ in our case, we indeed need a result that at least covers this range. Hence rewriting the elliptic-elliptic IVP   
\begin{align}\label{elel}
\begin{cases}
iu_t +\Delta u=c_1|u|^2u+c_2K(|u|^2)u\\ u(x,0)=u_0(x)\in H^s_x(\mathbb{R}^2), 
\end{cases}
\end{align}
with a multiplier operator $K$ given by \begin{align}\label{ell.K}K(f)=\mathcal{F}^{-1}\frac{\xi_1^2}{|\xi|^2} \mathcal{F}f
\end{align} for $\xi=(\xi_1,\xi_2)\neq 0$, the required global well-posedness result for \eqref{elel} and \eqref{ell.K} with a growth bound is as follows:

\begin{theorem}[See \cite{Yang}]\label{global}
Let $c_1+c_2>0.$ For any $1>s>\frac{2}{5}$, the initial value problem \eqref{elel} $\&$ \eqref{ell.K} is globally well-posed in $H^s(\mathbb{R}^2)$. Furthermore, there is a growth bound $$\sup_{t\in[0,T]}\norm{u(t)}_{H^s(\mathbb{R}^2)}\leq C(1+T)^{\frac{3s(1-s)}{2(5s-2)}+}$$ where the constant $C$ depends only on the index $s$, $\norm{u_0}_{L^2}$.

\end{theorem}
The essential ingredients of the proof of Theorem \ref{global} are the interaction Morawetz type estimate and the almost conservation of the modified energy obtained from plugging the smoothing of order $1-s$ operator $\widehat{Iu}(\xi)=m(\xi)\widehat{u}(\xi)$, \begin{align*}\label{multiplier}
m(\xi):=
\begin{cases}
\hspace{0.2cm}1, \hspace{1.5cm}|\xi|<N \\
\Big(\frac{|\xi|}{N}\Big)^{s-1}, \hspace{0.47cm} |\xi|>2N
\end{cases}
\end{align*} in the usual Hamiltonian energy of the equation \eqref{elel}. The globalizing technique used in \cite{Yang} is called the $I$-method, in view of the operator $I:H^s\rightarrow H^1$, which was introduced by Colliander--Keel--Staffilani--Takaoka--Tao \cite{Colliander}. 
Next by virtue of Theorem \ref{global}, we reserve the following growth bound to be used later.  For $s>\frac{1}{2}$, define $\beta(s)\geq\frac{3s(1-s)}{2(5s-2)}$ so that we have the a priori estimate 
\begin{equation}
\norm{u(t)}_{H^s}\lesssim \langle t \rangle^{\beta(s)}=:T(t) \label{T(t)}
\end{equation}
for some non-decreasing function $T(t)$ (which we need in iterating the local result), and where the implicit constant depends on the Sobolev index and $L^2$ norm of the initial data. 

We now consider the forced and weakly damped Davey--Stewartson equation \eqref{elelforced} that can be reduced to the single equation 
\begin{align}\label{forcedsingle}
iu_t +\Delta u+i\delta u=c_1|u|^2u+c_2K(|u|^2)u +f    
\end{align}
with the same multiplier operator $K$ as in \eqref{ell.K}. Below we list a few properties of $K$ to be used in the energy calculations of \eqref{forcedsingle}:
\begin{enumerate}\label{k}
    \item[(i)] $K$ is a bounded linear operator on $L^p$, $1<p<\infty$,
    \item[(ii)] $\overline{K(\psi)}=K(\overline{\psi})$,
    \item[(iii)] $\int K(\psi)\varphi=\int \psi K(\varphi)$.
\end{enumerate} As a consequence of smoothing estimate, we will study the existence of global atractors for the dissipative DS equation. For a general discussion, let $H$ be a phase space, and assuming the existence of a global wellposed problem in $H$, we also let $U(t): H\rightarrow H$ denote the evolution operator, mapping data to solution. Our goal is to describe the long time asymptotics of the solution by a compact invariant subset of $H$. The following definitions and result will give the features of this invariant subset which we need in establishing the proof of Theorem \ref{thm:attractor}.
\begin{Definition}[\cite{Temam}]
We say that a compact subset $\mathcal{A}$ of the phase space $H$ is a global attractor for the semigroup $\{U(t)\}_{t\geq 0}$ if $U(t)\mathcal{A}=\mathcal{A}$ for all $t>0$, and $$\lim\limits_{t\rightarrow \infty}d(U(t)g,\mathcal{A})=0,\,\,\forall g\in H.$$
\end{Definition}
Next we define a bounded subset of a phase space which is of particular importance, called an absorbing set, and all solutions enter into this set after a certain time. The existence of this set can be verified by using the energy estimates for the equation. 
\begin{Definition}[\cite{Temam}]
A bounded subset $\mathcal{B}$ of a phase space $H$ is called an absorbing set if for any bounded $\mathcal{S}\subset H$, there exists $T=T(\mathcal{S})$ such that $U(t)\mathcal{S}\subset \mathcal{B}$ for all $t\geq T$.
\end{Definition}
Note that a global attractor for a semigroup, if exists, implies the existence of an absorbing set whereas the converse of this argument holds still true whenever a semigroup $\{U(t)\}_{t\geq 0}$ satisfies certain criterion:
\begin{theorem}[\cite{Temam}]\label{attractortemam}
Suppose that $H$ is a metric space and $U(t):H\rightarrow H$ is a continuous semigroup defined for all $t\geq 0$. Furthermore suppose that there is an absorbing set $\mathcal{B}$. If the semigroup $\{U(t)\}_{t\geq 0}$ is asymptotically compact, that is, for every bounded sequence $\{x_k\}$ in $H$ and every sequence of times $t_k\rightarrow \infty$, the set $\{U(t_k)x_k\}_k$ is relatively compact in $H$, then the $\omega$-limit set $$\omega(\mathcal{B})=\bigcap\limits_{s\geq0}\overline{\bigcup\limits_{t\geq s} U(t)\mathcal{B}}$$
is a global attractor, where the closure is taken on $H$.
\end{theorem}
By the above theorem, the problem of existence of a global attractor boils down to finding an absorbing ball with which one needs to prove asymptotic compactness. 

The theory of existence of global attractors is only meaningful for systems as $t\rightarrow \infty$. Thus to see that the equation \eqref{elelforced} is globally wellposed in the energy space, we can proceed in a similar fashion to the proof of the existence of an absorbing set in section $2$ of \cite{Wang} to obtain the following a priori estimate: 
\begin{align}\label{globalH1}
\norm{u(t)}_{H^1}\leq Ae^{-Bt}+C,\,\,t>0,    \end{align}
where  $A=A(\norm{f}_{L^2},\norm{u_0}_{H^1})$, $B=B(\delta)>0$, $C=C(\delta, \norm{f}_{L^2})$. Note that \eqref{globalH1} implies the existence of an absorbing ball with radius $C=C(\delta, \norm{f}_{L^2})$ as well.

\section{Main Method}\label{method}
In this section we discuss Tao's $[k;Z]$ multipliers method \cite{Tao} so as to study the estimate of multilinear expression associated to the elliptic-elliptic type of Davey-Stewartson system. Suppose that $Z$ is any additive abelian group with an invariant measure $d\xi$, as an example, one can take $Z=\mathbb{R}^{n+1}$ with Lebesgue measure or $\mathbb{Z}^n\times \mathbb{R}$ with the product of counting and Lebesgue measures. Let $k\geq 2$ be any integer, and let $\Gamma_k(Z)\subset Z^k$ denote the hyperplane 
$$\Gamma_k(Z):=\{(\xi_1,..., \xi_k)\in Z^k : \xi_1+...+\xi_k=0\}$$ endowed with the measure $$\int_{\Gamma_k(Z)}f=\int_{Z^{k-1}}f(\xi_1,...,\xi_{k-1},-\xi_1-...-\xi_{k-1})\,\text{d}\xi_1...\text{d}\xi_{k-1}.$$ 
\begin{Definition}
 A complex valued function $m:\Gamma_k(Z)\rightarrow \mathbb{C}$ is called the $[k;Z]$ multiplier if the inequality
\begin{align*}
\Big|\int_{\Gamma_k(Z)}m(\xi_1,...,\xi_k)\prod_{i=1}^kf_i(\xi_i)\Big|\leq \norm{m}_{[k;Z]}\prod_{i=1}^k\norm{f_i}_{L^2(Z)}
\end{align*}
holds for all test functions $f_i$ on $Z$ and the best constant, denoted by $\norm{m}_{[k;Z]}$. 
\end{Definition}
Note that $\norm{\cdot}_{[k;Z]}$ determines a norm on the $[k;Z]$ multipliers $m$. Since  multilinear estimates might boil down to bilinear estimates of some sort (reason later), we emphasize the case $k=3$ which specifically associates to the bilinear estimates for our Schr\"odinger type equation. So we write
\begin{align}
&\xi_1+\xi_2+\xi_3=0,\ \ \ \ \tau_1+\tau_2+\tau_3=0,\label{xi,tau}\\& \lambda_j:= \tau_j+h_j(\xi_j), \ \ \ \ h_j(\xi_j)=\pm |\xi_j|^2 \ \ \text{for}\ \ 1\leq j \leq 3. \label{lambda,h}
\end{align}
Here $\lambda_j$ measures how close in frequency the $j^{th}$ wave is to a free solution. It is remarkable to note that for $d\geq 2$, whenever the two of the frequencies $\xi_1,\xi_2,\xi_3$ form an orthogonal pair, $\lambda_j$'s simultaneously vanish. Thus in order to take care of this situation we introduce the function $h: \Gamma_3(\mathbb{R}^2)\rightarrow \mathbb{R}$  defined by \begin{align}\label{resonancefnc}
h(\xi_1,\xi_2,\xi_3):= h_1(\xi_1)+h_2(\xi_2)+h_3(\xi_3)=\lambda_1+\lambda_2+\lambda_3
\end{align}
which measures to what extent the frequencies $\xi_1,\xi_2,\xi_3$
resonate with each other so let it be referred as a resonance function. The domain on which the resonance function may vanish depends on the sign of Schr\"odinger dispersion relation $h_j(\xi_j)=\pm|\xi_j|^2$. Up to symmetry we have two possibilities: the $(+++)$ case
\begin{align}\label{+++}
h_1(\xi)=h_2(\xi)= h_3(\xi)= |\xi|^2
\end{align} 
and the $(++-)$ case
\begin{align}\label{++-}
h_1(\xi)=h_2(\xi)= |\xi|^2, h_3(\xi)= -|\xi|^2.
\end{align} Comparing the two cases, analysis of the first one is rather easier since the resonance function \begin{align*}
h(\xi_1,\xi_2,\xi_3)=|\xi_1|^2+|\xi_2|^2+|\xi_3|^2
\end{align*} vanishes only at the origin. As for the second case, the resonance function \begin{align*}
h(\xi_1,\xi_2,\xi_3)=|\xi_1|^2+|\xi_2|^2-|\xi_3|^2
\end{align*}
can vanish whenever the frequencies $\xi_1$ and $\xi_2$ become orthogonal. 
More precisely, for $(\xi_1,\xi_2,\xi_3) \in \Gamma_3(\mathbb{R}^2)$, we write
\begin{align*}
|h(\xi_1,\xi_2,\xi_3)|=||\xi_1|^2+|\xi_2|^2-|\xi_3|^2|=2|\xi_1\cdot\xi_2|\sim |\xi_1||\xi_2||\pi/2-\angle (\xi_1,\xi_2)|
\end{align*} 
where $\angle (\xi_1,\xi_2)$ denotes the angle between $\xi_1$ and $\xi_2$. So the extent to which $\xi_1$ and $\xi_2$ get closer to being orthogonal, the more rapidly resonance function tends to vanish. At this point we assume that \begin{align} \label{h estimate}
|h(\xi_1,\xi_2,\xi_3)|\lesssim |\xi_1||\xi_2|,
\end{align}
and 
\begin{align} \label{angle estimate}
\angle(\xi_1,\xi_2)=\pi/2+O(\frac{|h(\xi_1,\xi_2,\xi_3)|}{|\xi_1||\xi_2|}).
\end{align}
We will estimate the $[3;\mathbb{R}^3]$ norm of multipliers by making use of the dyadic decomposition of the variables $\xi_j$, $\lambda_j$ and the function $h(\xi_1,\xi_2,\xi_3)$. Thus we use the capitalized variables $N_j$, $L_j$  and $H$ to denote the magnitude of the pieces into which the variables $\xi_j$, $\lambda_j$ and the resonance function $h$ decomposed respectively. Here these variables are assumed to be dyadic, that is, they range over the numbers of the form $2^k$, $k\in \mathbb{Z}$. In terms of sizes of the variables $N_j>0$, $j=1,2,3$, we write $N_{\min}\leq N_{\text{med}} \leq N_{\max}$ to denote the minimum, median and maximum of $N_1,N_2,N_3$. This, in its own right, saves us from repetitive analysis and reduces the number of cases substantially. Likewise we define $L_{\min}\leq L_{\text{med}} \leq L_{\max}$ for $L_j>0$, $j=1,2,3.$ Next we make some assumptions on the sizes of these variables. Before doing so, we need to state several lemmas from \cite{Tao}.
\begin{lemma}[Comparison Principle]\label{comp.prin.}
Let $m$ and $M$ be $[k;Z]$ multipliers, if for all $\xi\in \Gamma_k(Z)$ $|m(\xi)|\leq M(\xi)$, then  $\norm{m}_{[k;Z]}\leq \norm{M}_{[k;Z]}$. Furthermore if $a_1,..., a_k$ are real valued functions on $Z$ and $m$ is a $[k;Z]$ multiplier, then
\begin{align*}
\norm{m(\xi)\prod_{i=1}^ka_i(\xi_i)}_{[k;Z]}\leq \norm{m}_{[k;Z]}\prod_{i=1}^k\norm{a_i}_{L^{\infty}(Z)}.
\end{align*}
\end{lemma}
\begin{lemma}\label{av}
For any $\xi_0 \in \Gamma_k(Z)$ and any $[k;Z]$ multiplier $m$, we have the translation invariance of the norm
\begin{align}\label{trans.invar.}
\norm{m(\xi+\xi_0)}_{[k;Z]}=\norm{m(\xi)}_{[k;Z]}
\end{align}
also we have the averaging estimate 
\begin{align}\label{aver.}
\norm{m*\mu}_{[k;Z]}\leq \norm{m}_{[k;Z]}\norm{\mu}_{L^1(\Gamma_k(Z))}
\end{align}
for any finite measure $\mu$ on $\Gamma_k(Z).$
\end{lemma}
\begin{proof}[Proof]
Translation invariance immediately follows by the definition of a $[k;Z]$ multipliers. For the averaging estimate, for all test functions $f_i$, bearing in mind that $\mu$ is a finite measure on $\Gamma_k(Z)$ and then application of Minkowski integral inequality with $p=1$ leads to \begin{multline*}
\Big|\int_{\Gamma_k(Z)}m*\mu(\xi)\prod_{i=1}^kf_i(\xi_i)\Big|=\Big|\int_{\Gamma_k(Z)}\Big(\int_{\Gamma_k(Z)}m(\xi-\eta)\,\text{d}\mu(\eta)\Big)\prod_{i=1}^kf_i(\xi_i)\Big| \\ \leq \Big|\int_{\Gamma_k(Z)}\Big(\int_{\Gamma_k(Z)}m(\xi-\eta)\prod_{i=1}^kf_i(\xi_i)\Big)\,\text{d}\mu(\eta)\Big|
\end{multline*}
Next since $m$ is a $[k;Z]$ multipliers we estimate the inner integral by using the definition, and hence the total integral is majorized by
\begin{align*}
\prod_{i=1}^k \norm{f_i}_{L^2(Z)}\int_{\Gamma_k(Z)}\norm{m(\xi-\eta)}_{[k;Z]}\,\text{d}\mu(\eta)=\norm{m(\xi)}_{[k;Z]}\norm{\mu}_{L^1(\Gamma_k(Z))}\prod_{i=1}^k \norm{f_i}_{L^2(Z)}
\end{align*}
where we make use of \eqref{trans.invar.}  in the last line.

\end{proof} It is remarkable that by the finiteness assumption of the measure  $\mu$, the mapping  $m\rightarrow m*\mu$ can be regarded as an averaging of $m$. Therefore by an averaging over unit time scales, implementation of the lemmas above allow us to restrict the multiplier $m(\xi_1,\xi_2,\xi_3)$ to the region 
\begin{align*}
|\lambda_j|\gtrsim 1, \ \ j=1,2,3.
\end{align*}
Furthermore since we deal with a multiplier with no singularities for $|\xi_j|\ll1$, by using the similar reason we may assume that 
\begin{align*}
\max(|\xi_1|,|\xi_2|,|\xi_3|)\gtrsim 1
\end{align*}
Therefore, through a decomposition of the variables we may assume without loss of generality that
\begin{align*}
N_{\max}\gtrsim 1, \ \ L_{\min}\gtrsim 1.
\end{align*}
We now fix some summation conventions in order to be used in the rest. Any summation of the form $L_{\max}\sim...$ is a sum over the three dyadic variables $L_1,L_2,L_3$, for example,
\begin{align*}
\sum_{L_{\max}\sim H}:=\sum_{L_1,L_2,L_3\gtrsim 1: L_{\max}\sim H}
\end{align*}
Furthermore any summation of the form $N_{\max}\sim ...$ is a sum over the three dyadic variables $N_1,N_2,N_3>0$, for instance,
\begin{align*}
\sum_{N_{\max}\sim N_{\text{med}}\sim N}:=\sum_{N_1,N_2,N_3> 0: N_{\max}\sim N_{\text{med}}\sim N}
\end{align*}
Letting $m$ be a $[3;\mathbb{R}^2\times \mathbb{R}]$ multiplier, in the subsequent discussion, we intend to study the problem of controlling 
\begin{align}\label{3Zmult}
\norm{m((\xi_1,\tau_1),(\xi_2,\tau_2),(\xi_3,\tau_3))}_{[3;\mathbb{R}^2\times \mathbb{R}]},
\end{align} 
so by a dyadic decomposition of the support of $m$ in the variables $\xi_j$, $\lambda_j$ together with a dyadic decomposition of the resonance function $h$ we have
\begin{align}\label{normoftriplesum}
\eqref{3Zmult}\lesssim\norm{\sum_{N_{\max}\gtrsim 1}\sum_{H}\sum_{L_1,L_2,L_3}m((N_1,L_1),(N_2,L_2),(N_3,L_3))X_{N_1,N_2,N_3;H;L_1,L_2,L_3}}_{[3;\mathbb{R}^2\times \mathbb{R}]}
\end{align}
where $X_{N_1,N_2,N_3;H;L_1,L_2,L_3}$ is the multiplier 
\begin{align*}
X_{N_1,N_2,N_3;H;L_1,L_2,L_3}(\xi,\tau):=\chi_{|{h(\xi)|\sim H}}\prod_{j=1}^3\chi_{|\xi_j|\sim N_j}\chi_{|\lambda_j|\sim L_j}.
\end{align*}
Note that $N_j$ and $L_j$, in turn, measure the size of the frequency of the $j^{th}$ wave and how closely it approximates a free solution, whereas $H$ measures the amount of resonance.  
From \eqref{xi,tau}, \eqref{lambda,h} and \eqref{resonancefnc} it can be deduced that $X_{N_1,N_2,N_3;H;L_1,L_2,L_3}$ vanishes unless
\begin{align}\label{N_maxsimN_med}
N_{\max}\sim N_{\text{med}}
\end{align}
and 
\begin{align}\label{L_maxsimmax}
L_{\max}\sim \max(L_{\text{med}},H).
\end{align}
Therefore using \eqref{N_maxsimN_med}, \eqref{L_maxsimmax} and implementing Schur's test \cite{Tao} (which enables us to replace sum with a supremum) to the sums in $N_{\max}$ and $N_{\text{med}}$, we obtain
\begin{multline*}
\eqref{normoftriplesum}\lesssim\sup_{N\gtrsim 1}\Bigl\|\sum_{N_{\max}\sim N_{\text{med}}\sim N}\sum_{H}\sum_{L_{max}\sim \max{(L_{\text{med}},H)}}m((N_1,L_1),(N_2,L_2),(N_3,L_3))\\\times X_{N_1,N_2,N_3;H;L_1,L_2,L_3}\Bigl\|_{[3;\mathbb{R}^2\times \mathbb{R}]}.
\end{multline*} Therefore, by the triangle inequality and \eqref{L_maxsimmax} it suffices to control
\begin{multline}\label{doublesum}
\sum_{N_{\max}\sim N_{\text{med}}\sim N}\sum_{L_1,L_2,L_3\gtrsim1}m((N_1,L_1),(N_2,L_2),(N_3,L_3))\norm{X_{N_1,N_2,N_3;L_{\max};L_1,L_2,L_3}}_{[3;\mathbb{R}^2\times\mathbb{R}]}
\end{multline}
or
\begin{multline}\label{triplesum}
\sum_{N_{\max}\sim N_{\text{med}}\sim N}\sum_{L_{\max}\sim L_{\text{med}}}\sum_{H\ll L_{\max}}m((N_1,L_1),(N_2,L_2),(N_3,L_3))\\ \times\norm{X_{N_1,N_2,N_3;H;L_1,L_2,L_3}}_{[3;\mathbb{R}^2\times\mathbb{R}]}
\end{multline}
for all $N\gtrsim 1$.  
The following lemma gives a sharp bound for the quantity
\begin{align}\label{charblock}
\norm{X_{N_1,N_2,N_3;H;L_1,L_2,L_3}}_{[3;\mathbb{R}^2\times\mathbb{R}]}.
\end{align}
\begin{lemma}[\cite{Tao}]\label{cruciallem}
Suppose that $N_1,N_2,N_3>0$, $L_1,L_2,L_3>0$ and $H>0$ satisfy \eqref{N_maxsimN_med} and \eqref{L_maxsimmax}.
\begin{enumerate}
\item In the case $(+++)$, let the dispersion relations be given by \eqref{+++} so we may assume that $H\sim N_{\max}^2$. Then 

\begin{align}\label{lem2,0}
\eqref{charblock}\lesssim L_{\min}^{1/2}N_{\max}^{-1/2}N_{\min}^{1/2}\min(N_{\max}N_{\min},L_{\text{med}})^{1/2}
\end{align}
\item In the case $(++-)$, let the dispersion relations be given by \eqref{++-} and from \eqref{h estimate} $H\lesssim N_1N_2$. Then
\begin{itemize}
\item $((++)\ \text{case})$ If $N_1\sim N_2\gtrsim N_3$ then \eqref{charblock} vanishes unless $H\sim N_1^2$ in this case we have
\begin{align}\label{lem2,1}
\eqref{charblock}\lesssim L_{\min}^{1/2}N_{\max}^{-1/2}N_{\min}^{1/2}\min(N_{\max}N_{\min},L_{\text{med}})^{1/2}
\end{align}
\item $((+-)\ \text{coherence})$ If $N_1\sim N_3\gtrsim N_2$ and $H\sim L_2\gg L_1,L_3,N_2^2$ then
\begin{align}\label{lem2,2}
\eqref{charblock}\lesssim L_{\min}^{1/2}N_{\max}^{-1/2}N_{\min}^{1/2}\min(H,\frac{H}{N_{\min}^2}L_{\text{med}})^{1/2}
\end{align}
The same estimate holds with the roles of $1$ and $2$ reversed.
\item In all other cases, 
\begin{align}\label{lem2,3}
\eqref{charblock}\lesssim L_{\min}^{1/2}N_{\max}^{-1/2}N_{\min}^{1/2}\min(H,L_{\text{med}})^{1/2}\min(1,\frac{H}{N_{\min}^2})^{1/2}.
\end{align}
\end{itemize}

\end{enumerate}
\end{lemma}
Also the lemma below demonstrates that higher order multilinear estimates might be reduced to the lower ordered ones and by this means the analysis of the whole multiplier splits up.  

\begin{lemma}[\cite{Tao}]\label{split}
If $k_1,k_2\geq 1$ and $m_1$ and $m_2$ are functions on $Z^{k_1}$ and $Z^{k_2}$ respectively, then we have the composition estimate
\begin{multline}\label{comp}
\norm{m_1(\xi_1,...,\xi_{k_1})m_2(\xi_{k_1+1},...,\xi_{k_1+k_2})}_{[k_1+k_2;Z]}\\ \leq \norm{m_1(\xi_1,...,\xi_{k_1})}_{[k_1+1;Z]}\norm{m_2(\xi_1,...,\xi_{k_2})}_{[k_2+1;Z]}.
\end{multline}
In particular, for every $m:Z^k\rightarrow \mathbb{R}$ we have the $TT^*$ identity
\begin{align}\label{TT}
\norm{m(\xi_1,...,\xi_k)\overline{m(- \xi_{k+1},...,-\xi_{2k})}}_{[2k;Z]}=\norm{m(\xi_1,...,\xi_k)}_{[k+1;Z]}^2.
\end{align}

\end{lemma}

\begin{proof}[Proof]
We prove only the identity \eqref{comp}. So for $s=1,2$, let $L_s$ be the $k_s$-linear operator defined by
 \begin{align*}
 \langle L_s(f_1,...,f_{k_s}),f_{k_s+1} \rangle=\int_{\Gamma_{k_s+1}(Z)}m_s(\xi_1,...,\xi_{k_s})\prod_{j=1}^{k_s+1}f_j(\xi_j).
 \end{align*}
By duality this entails that
\begin{align}\label{ls}
\norm{L_s(f_1,...,f_{k_s})}_{L^2}=\sup_{\norm{f_{k_s+1}}_{L^2}=1}|\langle L_s(f_1,...,f_{k_s}),f_{k_s+1} \rangle|\leq \norm{m_s}_{[k_s+1;Z]}\prod_{j=1}^{k_s}\norm{f_j}_{L^2}.
\end{align}
Thus Cauchy-Schwarz inequality and \eqref{ls} lead to
\begin{multline*}
\int_{\Gamma_{k_1+k_2}(Z)}m_1(\xi_1,...,\xi_{k_1})m_2(\xi_{k_1+1},...,\xi_{k_1+k_2})\prod_{j=1}^{k_1+k_2}f_j(\xi_j)\\=\int_{\Gamma_{k_1+k_2}(Z)}m_1(\xi_1,...,\xi_{k_1})\prod_{j=1}^{k_1}f_j(\xi_j)\Big(m_2(\xi_{k_1+1},...,\xi_{k_1+k_2})\prod_{j=k_1+1}^{k_1+k_2}f_j(\xi_j)\Big)\\=\langle L_1(f_1,...,f_{k_1}),L_2(f_{k_1+1},...,f_{k_1+k_2}) \rangle  \leq \norm{L_1(f_1,...,f_{k_1})}_{L^2}\norm{L_2(f_{k_1+1},...,f_{k_1+k_2})}_{L^2}\\ \leq \norm{m_1}_{[k_1+1;Z]}\norm{m_2}_{[k_2+1;Z]}\prod_{j=1}^{k_1+k_2}\norm{ f_j}_{L^2}.
\end{multline*}
\end{proof}

\section{Proof of Proposition \ref{key}: Trilinear $X^{s,b}$ Estimate}\label{trilinearestimate}
Before getting start with the estimate, we suitably arrange the $X^{s,b}$ norm of the non-linearity so as to sort out the trilinear form.  By duality we may write the norm of the non-linearity as follows
\begin{multline*}
\norm{c_1|u|^2u+c_2K(|u|^2)u}_{X^{s+a,b-1}}=\\ \sup_{\norm{d}_{X^{-(s+a),1-b}}=1}\Big|\int_{\mathbb{R}^2\times \mathbb{R}}\big(c_1|u|^2u+c_2K(|u|^2)u\big)(x,t)d(x,t)\,\text{d}x\,\text{d}t\Big|
\end{multline*}
Writing Fourier transform of nonlinear pieces as a convolution, we have
 \begin{align*}
 \widehat{u\overline{v}w}(\xi,\tau)&=\widehat{u}*\widehat{\overline{v}}*\widehat{w}(\xi,\tau)=\underset{\substack{\xi_1 + \xi_2 + \xi_3-\xi=0 \\ \tau_1+\tau_2+\tau_3-\tau=0}}{\int} \widehat{u}(\xi_1,\tau_1)\widehat{\overline{v}}(\xi_2,\tau_2)\widehat{w}(\xi_3,\tau_3)\,\text{d}\xi_1\, \text{d}\xi_2\, \text{d}\tau_1\,d\tau_2 \\&= \underset{\substack{\xi_1 + \xi_2 + \xi_3-\xi=0 \\ \tau_1+\tau_2+\tau_3-\tau=0}}{\int} \widehat{u}(\xi_1,\tau_1)\overline{\widehat{v}(-\xi_2,-\tau_2)}\widehat{w}(\xi_3,\tau_3)\,\text{d}\xi_1\, \text{d}\xi_2\, \text{d}\tau_1\,d\tau_2, 
 \end{align*}
and the second term in the non-linearity is dealt as follows
  \begin{align*}
 \widehat{K(u\overline{v})w}(\xi,\tau)&=\alpha(\xi)\widehat{u\overline{v}}*\widehat{w}(\xi,\tau)=\underset{\substack{\xi_1 + \xi_2 - \xi=0 \\ \tau_1+\tau_2-\tau=0}}{\int} \alpha(\xi_1)\widehat{u\overline{v}}(\xi_1,\tau_1)\widehat{w}(\xi_2,\tau_2)\,\text{d} \xi_1\, \text{d} \tau_1\\&=\underset{\substack{\xi_1 + \xi_2 - \xi=0 \\ \tau_1+\tau_2-\tau=0}}{\int} \alpha(\xi_1)\Big[\underset{\substack{\xi_1 - \xi_3 - \xi_4=0 \\ \tau_1-\tau_3-\tau_4=0}}{\int}\widehat{u}(\xi_3,\tau_3)\widehat{\overline{v}}(\xi_4,\tau_4)\,\text{d}\xi_3\, \text{d} \tau_3 \Big]\widehat{w}(\xi_2,\tau_2)\,\text{d}\xi_1\, \text{d} \tau_1\\&=\underset{\substack{\xi_1+ \xi_2 - \xi=0 \\ \tau_1+\tau_2-\tau=0}}{\int} \alpha(\xi_1)\Big[\underset{\substack{\xi_1 - \xi_3 - \xi_4=0 \\ \tau_1-\tau_3-\tau_4=0}}{\int}\widehat{u}(\xi_3,\tau_3)\overline{\widehat{v}(-\xi_4,-\tau_4)}\,\text{d}\xi_3\, \text{d}\tau_3 \Big]\widehat{w}(\xi_2,\tau_2)\,\text{d}\xi_1\, \text{d} \tau_1 \\&=\underset{\substack{\xi_2 +\xi_3 + \xi_4-\xi=0 \\ \tau_2+\tau_3 +\tau_4-\tau=0}}{\int} \alpha(\xi_3+\xi_4)\widehat{w}(\xi_2,\tau_2)\widehat{u}(\xi_3,\tau_3)\overline{\widehat{v}(-\xi_4,-\tau_4)}\,\text{d}\xi_3\,\text{d}\tau_3\,\text{d}\xi_4\, \text{d}\tau_4.  
 \end{align*}
In the last line we have taken account of the relations regarding the integration region so as to replace $\text{d}\xi_1\,\text{d}\tau_1$ by $\text{d}\xi_4\,\text{d}\tau_4$. Rename the variables by writing $\xi_1,\xi_2,\xi_3$\,; $\tau_1,\tau_2,\tau_3$ in substitution for $\xi_3,\xi_4,\xi_2$\,; $\tau_3,\tau_4,\tau_2$ respectively to get that
\begin{align*}
\widehat{K(u\overline{v})w}(\xi,\tau)=\underset{\substack{\xi_1 +\xi_2 + \xi_3-\xi=0 \\ \tau_1+\tau_2 +\tau_3-\tau=0}}{\int} \alpha(\xi_1+\xi_2)\widehat{u}(\xi_1,\tau_1)\overline{\widehat{v}(-\xi_2,-\tau_2)}\widehat{w}(\xi_3,\tau_3)\,\text{d}\xi_1\,\text{d}\tau_1\,\text{d}\xi_2\, \text{d}\tau_2. 
\end{align*}
Therefore, as Fourier transform preserves the inner product we obtain
\begin{multline*}
\int_{\mathbb{R}^2\times \mathbb{R}}\big(c_1u\overline{v}w+c_2K(u\overline{v})w\big)(x,t)d(x,t)\,\text{d}x\,\text{d}t\\=\int_{\mathbb{R}^2\times \mathbb{R}} (c_1\mathcal{F}(u\overline{v}w)+c_2\mathcal{F}( K(u\overline{v})w))(\xi,\tau)\mathcal{F}^{-1}d(\xi,\tau)\,\text{d}\xi\, \text{d}\tau\\=\underset{\substack{\xi_1 +\xi_2 + \xi_3-\xi=0 \\ \tau_1+\tau_2 +\tau_3-\tau=0}}{\int}[c_1+ c_2\alpha(\xi_1+\xi_2)]\widehat{u}(\xi_1,\tau_1)\overline{\widehat{v}(-\xi_2,-\tau_2)}\widehat{w}(\xi_3,\tau_3)\widehat{d}(-\xi,-\tau)\,\text{d}\xi_1\,\text{d}\tau_1\,\text{d}\xi_2\, \text{d}\tau_2\,\text{d}\xi\, \text{d}\tau\\= \underset{\substack{\xi_1 +\xi_2 + \xi_3+\xi_4=0 \\ \tau_1+\tau_2 +\tau_3+\tau_4=0}}{\int}[c_1+ c_2\alpha(\xi_1+\xi_2)]\widehat{u}(\xi_1,\tau_1)\overline{\widehat{v}(-\xi_2,-\tau_2)}\widehat{w}(\xi_3,\tau_3)\widehat{d}(\xi_4,\tau_4)\,\text{d}\xi_1\,\text{d}\xi_2\,\text{d}\xi_3\,\text{d}\tau_1\,\text{d}\tau_2\, \text{d}\tau_3
\end{multline*}
where, in the last equality, first we replace $\text{d}\xi\, \text{d}\tau$ with  $\text{d}\xi_3\,\text{d}\tau_3$ by using the relations of integral region and then  swap the variables $\xi$, $\tau$ with $-\xi_4$, $-\tau_4$ respectively. By using the fact that the spaces  $X^{s,b-1}_{\tau=h(\xi)}$ and $X^{-s,1-b}_{\tau=-h(-\xi)}$ are dual to one another, we define
\begin{align*}
&f_1(\xi_1,\tau_1)=\langle\xi_1\rangle^s\langle\tau_1+|\xi_1|^2\rangle^b|\widehat{u}(\xi_1,\tau_1)|\\& f_2(\xi_2,\tau_2)=\langle\xi_2\rangle^s\langle\tau_2-|\xi_2|^2\rangle^b|\widehat{v}(-\xi_2,-\tau_2)|
\\&f_3(\xi_3,\tau_3)=\langle\xi_3\rangle^s\langle\tau_3+|\xi_3|^2\rangle^b|\widehat{w}(\xi_3,\tau_3)|
\\&f_4(\xi_4,\tau_4)=\langle\xi_4\rangle^{s+a}\langle\tau_4-|\xi_4|^2\rangle^{1-b}|\widehat{d}(\xi_4,\tau_4)|\end{align*}
and 
\begin{multline*} m((\xi_1,\tau_1),(\xi_2,\tau_2),(\xi_3,\tau_3),(\xi_4,\tau_4))\\=\frac{[c_1+c_2\alpha(\xi_1+\xi_2)]\langle\xi_4\rangle^{s+a}}{\prod_{j=1}^3\langle\xi_j\rangle^s \langle\tau_1+|\xi_1|^2\rangle^b\langle\tau_2-|\xi_2|^2\rangle^b\langle\tau_3+|\xi_3|^2\rangle^b\langle\tau_4-|\xi_4|^2\rangle^{1-b}},
\end{multline*} thus the required estimate takes the form
\begin{align*}
\Big|\int_{\Gamma_4(\mathbb{R}^2\times \mathbb{R})}  m((\xi_1,\tau_1),(\xi_2,\tau_2),(\xi_3,\tau_3),(\xi_4,\tau_4))\prod_{j=1}^{4}f_j(\xi_j,\tau_j) \Big|\leq \norm{m}_{[4;\mathbb{R}^2\times \mathbb{R}]}\prod_{j=1}^{4}\norm{f_j}_{L^2(\mathbb{R}^2\times\mathbb{R})}
\end{align*}
where $$\Gamma_4(\mathbb{R}^2\times \mathbb{R})=\{ ((\xi_1,\tau_1),(\xi_2,\tau_2),(\xi_3,\tau_3),(\xi_4,\tau_4))\in (\mathbb{R}^2\times\mathbb{R})^4:\sum_{j=1}^4(\xi_j,\tau_j)=0 \}.$$
Hence it suffices to show that  \begin{align*}
\norm{m((\xi_1,\tau_1),(\xi_2,\tau_2),(\xi_3,\tau_3),(\xi_4,\tau_4))}_{[4;\mathbb{R}^2\times \mathbb{R}]}\lesssim1.
\end{align*} 
\begin{proof}[Proof of Proposition \ref{key}]
We may suppose without loss of generality that $$|\xi_4|\sim \max_{1\leq j \leq 4}|\xi_j|$$
because other cases are easier and follow immediately. In this case the structure of the hyperplane $\Gamma_4(\mathbb{R}^2\times \mathbb{R})$ suggests three cases to consider:
 
 \begin{itemize}
\item \textbf{Case 1} $|\xi_4|\sim |\xi_1|$
 \item \textbf{Case 2} $|\xi_4|\sim |\xi_2|$
\item \textbf{Case 3} $|\xi_4|\sim |\xi_3|$
\end{itemize}
We begin with examining the Case $2$. In this case, the multiplier is estimated by
\begin{multline*}
m((\xi_1,\tau_1),(\xi_2,\tau_2),(\xi_3,\tau_3),(\xi_4,\tau_4))\lesssim \frac{\langle\xi_4 \rangle^a \langle \xi_3 \rangle^{-s}}{\langle\tau_3+|\xi_3|^2\rangle^b \langle \tau_4-|\xi_4|^2\rangle^{1-b}} \times\frac{\langle \xi_1 \rangle^{-s}}{\langle\tau_1+|\xi_1|^2\rangle^b \langle \tau_2-|\xi_2|^2\rangle^{b}}\\=:m_{3,4}((\xi_3,\tau_3),(\xi_4,\tau_4))\times m_{1,2}((\xi_1,\tau_1),(\xi_2,\tau_2)).
\end{multline*} 
Using Lemma \ref{comp.prin.} and Lemma \ref{split} we have the bound
\begin{multline*}
||m((\xi_1,\tau_1),(\xi_2,\tau_2),(\xi_3,\tau_3),(\xi_4,\tau_4))||_{[4;\mathbb{R}^2\times \mathbb{R}]}\lesssim||m_{1,2}((\xi_1,\tau_1),(\xi_2,\tau_2))||_{[3;\mathbb{R}^2\times \mathbb{R}]}\\ \times||m_{3,4}((\xi_3,\tau_3),(\xi_4,\tau_4))||_{[3;\mathbb{R}^2\times \mathbb{R}]}.
\end{multline*}
As a result the proof reduces to showing the following bilinear inequalities
\begin{align}
&\norm{m_{1,2}((\xi_1,\tau_1),(\xi_2,\tau_2))}_{[3;\mathbb{R}^2\times \mathbb{R}]}\lesssim 1, \label{1,2}\\& \norm{m_{3,4}((\xi_3,\tau_3),(\xi_4,\tau_4))}_{[3;\mathbb{R}^2\times \mathbb{R}]}\lesssim 1.\label{3,4}
\end{align}
At this juncture we introduce the dummy variables $\xi_{\text{d}2}$, $\xi_{\text{d}3}$,  $\tau_{\text{d}2}$, $\tau_{\text{d}3}$ satisfying
\begin{align} 
&(\xi_1,\tau_1)+(\xi_2,\tau_2)+(\xi_{\text{d}3},\tau_{\text{d}3})=0,\label{dummy36}\\&(\xi_{\text{d}2},\tau_{\text{d}2})+(\xi_3,\tau_3)+(\xi_4,\tau_4)=0 \label{dummy37}
\end{align}
so that \eqref{dummy36} and \eqref{dummy37} correspond to the hyperplanes with regards to \eqref{1,2} and \eqref{3,4} respectively. As discussed in the previous section, after decomposing the variables $\xi_{j_i}$,  $\lambda_{j_i}=\tau_{j_i}+h_{j_i}(\xi_{j_i})$ and the resonance function $h(\xi_{j_1},\xi_{j_2},\xi_{j_3})=h_{j_1}(\xi_{j_1})+h_{j_2}(\xi_{j_2})+h_{j_3}(\xi_{j_3})$ into dyadics $N_{j_i}$, $L_{j_i}$ and $H$ respectively, it suffices to establish the following inequalities
\begin{align*}
\sum_{N_{max}\sim N_{med}\sim N}\sum_{L_{max}\sim L_{med}}\sum_{H\ll L_{max}}m_{i_1,i_2}((\xi_{i_1},\tau_{i_1}),(\xi_{i_2},\tau_{i_2}))\norm{X_{{N_{\text{d}j}},N_{i_1},N_{i_2};H;L_{\text{d}j},L_{i_1},L_{i_2}}}_{[3;\mathbb{R}^2\times \mathbb{R}]}\lesssim 1
\end{align*}
and 
\begin{align*}
\sum_{N_{max}\sim N_{med}\sim N}\sum_{L_{\text{d}j},L_{i_1},L_{i_2}\gtrsim1}m_{i_1,i_2}((\xi_{i_1},\tau_{i_1}),(\xi_{i_2},\tau_{i_2}))\norm{X_{{N_{\text{d}j}},N_{i_1},N_{i_2};L_{max};L_{\text{d}j},L_{i_1},L_{i_2}}}_{[3;\mathbb{R}^2\times \mathbb{R}]}\lesssim 1
\end{align*}
for all $N\gtrsim 1$ and where, in turn, $(i_1,i_2,j)=(1,2,3), (3,4,2)$ associate to the bilinear inequalities \eqref{1,2} and \eqref{3,4}. Therefore starting with demonstrating \eqref{3,4}, we have to handle four cases in terms of magnitudes of the dyadic variables corresponding to relevant frequency variables:
\begin{enumerate}
  \item $N_{\text{d}2}\sim N_3\sim N_4$
  \item $ N_3\sim N_4\gg N_{\text{d}2}$
  \item $ N_{\text{d}2}\sim N_3 \gg N_4$
  \item $N_{\text{d}2}\sim N_4 \gg N_3$.
  
\end{enumerate} 
In the first of these cases $N_{\min}\sim N_{\max}\sim N \gtrsim1$, so by the estimate \eqref{lem2,3} of Lemma \ref{cruciallem} we have \begin{align*}\norm{X_{{N_{\text{d}2}},N_3,N_4;H;L_{\text{d}2},L_3,L_4}}_{[3;\mathbb{R}^2\times \mathbb{R}]}\lesssim  L_{\min}^{1/2}N^{-1/2}N^{1/2}&\min(H,L_{\text{med}})^{1/2}\min(1,\frac{H}{N^2})^{1/2}&\\ \lesssim L_{\min}^{1/2}H^{2\epsilon}L_{\text{med}}^{1/2-2\epsilon}.
\end{align*} 
Thus
\begin{align*}
\frac{\langle N_4 \rangle^a \langle N_3 \rangle^{-s}}{L_3^b L_4^{1-b} }\norm{X_{{N_{\text{d}2}},N_3,N_4;H;L_{\text{d}2},L_3,L_4}}_{[3;\mathbb{R}^2\times \mathbb{R}]}\lesssim\frac{N^{-s+a}L_{\min}^{1/2}L_{\text{med}}^{1/2 -2\epsilon}H^{2\epsilon}}{L_{\min}^{{1/2}+\epsilon} L_{\text{med}}^{1/2-\epsilon}}=\frac{N^{-s+a}H^{2\epsilon}}{L_{\min}^{\epsilon}L_{{\text{med}}}^{\epsilon}}.
\end{align*}
It follows that for if $H\sim L_{\max}$, then since $H\lesssim N_{\max}^2$, we have
\begin{multline*}
\sum_{N_{max}\sim N_{med}\sim N}\sum_{L_{\text{d}2},L_{3},L_{4}\gtrsim1}\frac{N^{-s+a}H^{2\epsilon}}{L_{\min}^{\epsilon}L_{{\text{med}}}^{\epsilon}} \ \ \sim  \sum_{N_{max}\sim N_{med}\sim N}\sum_{L_{\text{d}2},L_{3},L_{4}\gtrsim1}\frac{N^{-s+a}H^{3\epsilon}}{L_{\min}^{\epsilon}L_{{\text{med}}}^{\epsilon}L_{\max}^{\epsilon}} \\ \lesssim \sum_{N_{max}\sim N_{med}\sim N}\sum_{L_{\text{d}2},L_{3},L_{4}\gtrsim1}\frac{N^{-s+a+6\epsilon}}{L_{\min}^{\epsilon}L_{{\text{med}}}^{\epsilon}L_{\max}^{\epsilon}} \lesssim 1
\end{multline*}
provided $a<s$. But if $H\ll L_{\max}$ then $L_{\text{med}}\sim L_{\max}$, so summing in $H$ first we get 
 
\begin{multline*}
\sum_{N_{max}\sim N_{med}\sim N}\sum_{L_{max}\sim L_{med}}\sum_{H\ll L_{max}} \frac{N^{-s+a}H^{2\epsilon}}{L_{\min}^{\epsilon}L_{{\text{med}}}^{\epsilon}} \lesssim \sum_{N_{max}\sim N_{med}\sim N}\sum_{L_{max}\sim L_{med}} \frac{N^{-s+a+4\epsilon}}{L_{\min}^{\epsilon}L_{{\text{med}}}^{\epsilon/2}L_{\max}^{\epsilon/2}} \lesssim 1
\end{multline*}
provided again that $a<s$. 
For the second case, we write $N_3\sim N_4 \sim N_{\max}\sim N_{\text{med}}$ and $N_{\text{d}2}\sim N_{\min}$. We consider estimates \eqref{lem2,2} and \eqref{lem2,3} in Lemma \ref{cruciallem} in this case. Since $N_3\sim N_4 \sim N \gtrsim 1 $, we may establish the estimates corresponding to the estimate \eqref{lem2,3} by following just the same lines  of the  previous case. So it suffices to make use of the estimate \eqref{lem2,2} merely, where we have $H\sim L_{\text{d}2}\gg L_3,L_4,N_{\text{d}2}^2\sim N_{\min}^2.$
\begin{align*}\norm{X_{{N_{\text{d}2}},N_3,N_4;L_{\max};L_{\text{d}2},L_3,L_4}}_{[3;\mathbb{R}^2\times \mathbb{R}]}\lesssim  L_{\min}^{1/2}N^{-1/2}N_{\min}^{1/2}\min (H,\frac{H}{N_{\min}^2}L_{\text{med}})^{1/2} 
\end{align*} 
by which it follows that
\begin{multline*}
\sum_{N_{max}\sim N_{med}\sim N}\sum_{L_{\text{d}2},L_{3},L_{4}\gtrsim1}\frac{\langle N_4 \rangle^a \langle N_3 \rangle^{-s}}{L_3^b L_4^{1-b} }\norm{X_{{N_{\text{d}2}},N_3,N_4;L_{\max};L_{\text{d}2},L_3,L_4}}_{[3;\mathbb{R}^2\times \mathbb{R}]}\\ \lesssim \sum_{N_{max}\sim N_{med}\sim N}\sum_{L_{\text{d}2},L_{3},L_{4}\gtrsim1}\frac{N^{-s+a}L_{\min}^{1/2}N^{-1/2}N_{\min}^{1/2}\min (H,\frac{H}{N_{\min}^2}L_{\text{med}})^{1/2} }{L_{\min}^{1/2+\epsilon} L_{\text{med}}^{1/2-\epsilon} } \\ \lesssim\sum_{N_{max}\sim N_{med}\sim N}\sum_{L_{\text{d}2},L_{3},L_{4}\gtrsim1}\frac{N^{-s+a-1/2}N_{\min}^{1/2}(H^{1/2+\epsilon} (\frac{H L_{\text{med}}}{N_{\min}^2})^{1/2-\epsilon})^{1/2}}{L_{\min}^{\epsilon} L_{\text{med}}^{1/2-\epsilon} }\\=\sum_{N_{max}\sim N_{med}\sim N}\sum_{L_{\text{d}2},L_{3},L_{4}\gtrsim1}\frac{N^{-s+a-1/2}N_{\min}^{\epsilon}H^{1/2}}{L_{\min}^{\epsilon}L_{\text{med}}^{1/4-\epsilon/2} }.
\end{multline*}
Since $L_{\max}\sim H\lesssim N^2$ and $N\gtrsim1$, using $$1=\frac{L_{\max}^{\epsilon/2}\langle N_{\min}\rangle^{2\epsilon}}{L_{\max}^{\epsilon/2}\langle N_{\min}\rangle^{2\epsilon}}\lesssim\frac{N^{3\epsilon}}{L_{\max}^{\epsilon/2}\langle N_{\min}\rangle^{2\epsilon}}$$
in majorizing the sum above we obtain the bound
\begin{align*}
\sum_{N_{max}\sim N_{med}\sim N}\sum_{L_{\text{d}2},L_{3},L_{4}\gtrsim1}\frac{N^{-s+a+1/2+3\epsilon}N_{\min}^{\epsilon}}{L_{\min}^{\epsilon}L_{\text{med}}^{1/4-\epsilon/2} L_{\max}^{\epsilon/2}\langle N_{\min}\rangle^{2\epsilon}}\lesssim1
\end{align*}
which is possible as long as $a<s-1/2$. Note at this point that for a non-trivial  smoothing argument it is necessary to make the assumption that $s>1/2$. 
In the third and fourth cases, due to its effect in the use of Lemma \ref{cruciallem}, we have to decide the sign in the quantity $\lambda_{\text{d}2}=\tau_{\text{d}2}\pm|\xi_{\text{d}2}|^2$. No selection would lose the generality though, we prefer to set $\lambda_{\text{d}2}=\tau_{\text{d}2}+|\xi_{\text{d}2}|^2$. Hence for the third case, by this choice of the dummy modulation variable, we fall under the $(++)$ case which leads to $H\sim N_{\max}^2$. Thus the estimate \eqref{lem2,1} in Lemma \ref{cruciallem} is to be used. Once $H\sim L_{\max}$,
\begin{align*}
 \norm{X_{{N_{\text{d}2}},N_3,N_4;L_{\max};L_{\text{d}2},L_3,L_4}}_{[3;\mathbb{R}^2\times \mathbb{R}]}&\lesssim L_{\min}^{1/2}N^{-1/2}N_{\min}^{1/2}\min(NN_{\min},L_\text{med})^{1/2}\\&\lesssim L_{\min}^{1/2}(N^{3\epsilon}N_{\min}^{3\epsilon}L_{\text{med}}^{1-3\epsilon})^{1/2}=L_{\min}^{1/2}L_{\text{med}}^{1/2-3\epsilon/2}N^{3\epsilon/2}N_{\min}^{3\epsilon/2}
 \end{align*}
implies that
\begin{multline*}
\sum_{N_{max}\sim N_{med}\sim N}\sum_{L_{\text{d}2},L_{3},L_{4}\gtrsim1}\frac{\langle N_4 \rangle^a \langle N_3 \rangle^{-s}}{L_3^b L_4^{1-b} }\norm{X_{{N_{\text{d}2}},N_3,N_4;L_{\max};L_{\text{d}2},L_3,L_4}}_{[3;\mathbb{R}^2\times \mathbb{R}]}\\ \lesssim \sum_{N_{max}\sim N_{med}\sim N}\sum_{L_{\text{d}2},L_{3},L_{4}\gtrsim1}\frac{{N}^{-s+a+3\epsilon/2}N_{\min}^{3\epsilon/2}}{L_{min}^{\epsilon}L_{\text{med}}^{\epsilon/2}}.
\end{multline*}
Since $L_{\max}\sim H\sim N^2$, using the inequality $1=\frac{L_{\max}^{\epsilon/2}\langle N_{\min}\rangle^{2\epsilon}}{L_{\max}^{\epsilon/2}\langle N_{\min}\rangle^{2\epsilon}}\lesssim \frac{N^{3\epsilon}}{L_{\max}^{\epsilon/2}\langle N_{\min}\rangle^{2\epsilon}}$ in the above estimate we attain
\begin{align*}
\sum_{N_{max}\sim N_{med}\sim N}\sum_{L_{\text{d}2},L_{3},L_{4}\gtrsim1}\frac{{N}^{-s+a+9\epsilon/2}N_{\min}^{3\epsilon/2}}{L_{min}^{\epsilon}L_{\text{med}}^{\epsilon/2}L_{\max}^{\epsilon/2}\langle N_{\min}\rangle^{2\epsilon}}\lesssim 1
\end{align*}
as long as $a<s$. Otherwise if $H\ll L_{\max}$, then as $H\sim N^2$ and $L_{\text{med}}\sim L_{\max}$
\begin{multline*}
\sum_{N_{max}\sim N_{med}\sim N}\sum_{L_{max}\sim L_{med}}\sum_{H\ll L_{max}} \frac{N^{-s+a+3\epsilon/2}N_{\min}^{3\epsilon/2}}{L_{\min}^{\epsilon}L_{{\text{med}}}^{\epsilon/2}}\\ \sim \sum_{N_{max}\sim N_{med}\sim N}\sum_{L_{max}\sim L_{med}}\sum_{H\ll L_{max}} \frac{N^{-s+a+\epsilon}H^{\epsilon/4}N_{\min}^{3\epsilon/2}}{L_{\min}^{\epsilon}L_{{\text{med}}}^{\epsilon/8}L_{\max}^{3\epsilon/8}}.
\end{multline*} 
Summing this in $H$ we get the bound
\begin{align*}
\sum_{N_{max}\sim N_{med}\sim N}\sum_{L_{max}\sim L_{med}} \frac{N^{-s+a+\epsilon}N_{\min}^{3\epsilon/2}}{L_{\min}^{\epsilon}L_{{\text{med}}}^{\epsilon/8}L_{\max}^{\epsilon/8}}\lesssim 1
\end{align*}
so long as $a<s$. In the last case, we have $N_{\text{d}2}\sim N_4\sim N_{\max} \sim  N$ and $N_3\sim N_{\min}$. Sign analysis of the modulations $\lambda_{\text{d}2}$ (which is set in the previous case) and $\lambda_4$ suggests utilizing \eqref{lem2,2} and \eqref{lem2,3} in Lemma \ref{cruciallem}. In the separate case $H\sim L_{\max} \sim L_3 \gg L_{\text{d}2}, L_4, N_3^2$,  the bound \eqref{lem2,2}  gives rise to the estimate
\begin{multline*}
\sum_{N_{max}\sim N_{med}\sim N}\sum_{L_{\text{d}2},L_{3},L_{4}\gtrsim1}\frac{\langle N_4 \rangle^a \langle N_3 \rangle^{-s}}{L_3^b L_4^{1-b} }\norm{X_{{N_{\text{d}2}},N_3,N_4;L_{\max};L_{\text{d}2},L_3,L_4}}_{[3;\mathbb{R}^2\times \mathbb{R}]}\\ \lesssim \sum_{N_{max}\sim N_{med}\sim N}\sum_{L_{\text{d}2},L_{3},L_{4}\gtrsim1}\frac{N^{a-1/2}N_{\min}^{1/2}L_{\min}^{1/2}\min({H,\frac{H}{N_{\min}^2}L_{\text{med}}})^{1/2}}{\langle N_{\min} \rangle^s L_{\max}^{1/2+\epsilon}L_{\min}^{1/2-\epsilon}} \\ \lesssim \sum_{N_{max}\sim N_{med}\sim N}\sum_{L_{\text{d}2},L_{3},L_{4}\gtrsim1}\frac{N^{a-1/2}N_{\min}^{1/2} H^{1/2}L_{\min}^{\epsilon}}{\langle N_{\min} \rangle^s L_{\max}^{1/2+\epsilon}}.
\end{multline*}
Since $L_{\max}\sim H \lesssim N^2$ writing $H\lesssim L_{\max}^{1-\epsilon}N^{2\epsilon}$ we bound the above sum by
\begin{multline*}
\sum_{N_{max}\sim N_{med}\sim N}\sum_{L_{\text{d}2},L_{3},L_{4}\gtrsim1}\frac{N^{a-1/2+\epsilon}N_{\min}^{1/2} L_{\min}^{\epsilon}}{\langle N_{\min} \rangle^s L_{\max}^{3\epsilon/2}}\\\lesssim \sum_{N_{max}\sim N_{med}\sim N}\sum_{L_{\text{d}2},L_{3},L_{4}\gtrsim1}\frac{N^{a-1/2+\epsilon}N_{\min}^{\epsilon} }{\langle N_{\min} \rangle^{s-1/2+\epsilon} L_{\max}^{\epsilon/6}L_{\text{med}}^{\epsilon/6}L_{\min}^{\epsilon/6}}
\end{multline*}
which is summable provided that $a<1/2$. For the situations with  $L_{\max}\gg H$ where the estimate \eqref{lem2,3} is available, we have the bound 
\begin{multline*}
\sum_{N_{max}\sim N_{med}\sim N}\sum_{L_{max}\sim L_{med}}\sum_{H\ll L_{max}}\frac{\langle N_4 \rangle^a \langle N_3 \rangle^{-s}}{L_3^b L_4^{1-b} }\norm{X_{{N_{\text{d}2}},N_3,N_4;H;L_{\text{d}2},L_3,L_4}}_{[3;\mathbb{R}^2\times \mathbb{R}]}\\ \lesssim 
\sum_{N_{max}\sim N_{med}\sim N}\sum_{L_{max}\sim L_{med}}\sum_{H\ll L_{max}} \frac{N^{a-1/2}N_{\min}^{1/2}L_{\min}^{1/2}\min({H,L_{\text{med}}})^{1/2}\min({1,\frac{H}{N_{\min}^2})^{1/2}}}{\langle N_{\min} \rangle^s L_{\min}^{1/2+\epsilon}L_{\text{med}}^{1/2-\epsilon}}\\
\lesssim \sum_{N_{max}\sim N_{med}\sim N}\sum_{L_{max}\sim L_{med}}\sum_{H\ll L_{max}} \frac{N^{a-1/2}N_{\min}^{\epsilon}H^{2\epsilon}L_{\text{med}}^{1/2-2\epsilon}}{\langle N_{\min} \rangle^{s-1/2+\epsilon} L_{\min}^{\epsilon}L_{\text{med}}^{1/2-\epsilon}}
\end{multline*}
as $H\lesssim N^2$ and $L_{\max}\sim L_{\text{med}}$, after summing in $H$, this is bounded by
\begin{align*}
\sum_{N_{max}\sim N_{med}\sim N}\sum_{L_{max}\sim L_{med}} \frac{N^{a-1/2+4\epsilon}N_{\min}^{\epsilon}}{\langle N_{\min} \rangle^{s-1/2+\epsilon} L_{\min}^{\epsilon}L_{\text{med}}^{\epsilon/2}L_{\max}^{\epsilon/2}}\lesssim 1
\end{align*}
provided that $a<1/2$. 
Also for the instances with $L_{\max}\sim H$ where the estimate \eqref{lem2,3} is available, we proceed similarly as above to obtain the bound
\begin{align*}
 \sum_{N_{max}\sim N_{med}\sim N}\sum_{L_{\text{d}2},L_{3},L_{4}\gtrsim1}\frac{N^{a-1/2+5\epsilon}N_{\min}^{\epsilon}}{\langle N_{\min}\rangle^{s-1/2+\epsilon}  L_{\min}^{\epsilon}L_{\text{med}}^{\epsilon/2}L_{\max}^{\epsilon}} \lesssim 1
\end{align*}
provided that $a>1/2$. This completes the proof of \eqref{3,4}. The proof of \eqref{1,2} is the repetition of done for \eqref{3,4} without the multiplier $\langle\xi_4 \rangle^a$ and follows by using \eqref{dummy36} and making the assumption $s>1/2$. Next for the Case $1$, we have

\begin{multline*}
m((\xi_1,\tau_1),(\xi_2,\tau_2),(\xi_3,\tau_3),(\xi_4,\tau_4))\lesssim \frac{\langle\xi_4 \rangle^a \langle \xi_2\rangle^{-s}}{\langle\tau_2-|\xi_2|^2\rangle^b \langle \tau_4-|\xi_4|^2\rangle^{1-b}} \times\frac{\langle \xi_3 \rangle^{-s}}{\langle\tau_1+|\xi_1|^2\rangle^b \langle \tau_3+|\xi_3|^2\rangle^{b}}\\=:m_{2,4}((\xi_2,\tau_2),(\xi_4,\tau_4))\times m_{1,3}((\xi_1,\tau_1),(\xi_3,\tau_3)).
\end{multline*} 
We pick dummy variables $\xi_{\text{d}j}$, $\tau_{\text{d}j}$ for $j=2,3$ such that
\begin{align*}
&(\xi_2,\tau_2)+(\xi_{\text{d}3},\tau_{\text{d}3})+(\xi_4,\tau_4)=0 \\& (\xi_1,\tau_1)+(\xi_{\text{d}2},\tau_{\text{d}2})+(\xi_3,\tau_3)=0
\end{align*}
with which we write $\lambda_{\text{d}2}=\tau_{\text{d}2}+|\xi_{\text{d}2}|^2$ and $\lambda_{\text{d}3}=\tau_{\text{d}3}-|\xi_{\text{d}3}|^2$. So then analysis of these two multipliers fall into $(+++)$ case which is substantially easier and only the estimate \eqref{lem2,0} for \eqref{charblock}  is taken for. In the spirit of the analysis of the subcase $(3)$ above, $[3;\mathbb{R}^3]$ norms of the two multipliers $m_{1,3}$ and $m_{2,4}$ can be shown to be finite provided that $a<1/2$ and $s>1/2.$ Lastly, Case $3$ immediately follows from Case $1$ because the variables $\xi_1$ and $\xi_3$ appear to be symmetric.
\end{proof} 
In general optimality in the multilinear estimates  can be exhibited by Knapp-type examples as such in \cite{Kenig}, \cite{Compaan}. In what follows we construct a suitable such example for which the trilinear estimate \eqref{keyest} fails for $a>1/2$, as a result a half derivative gain is the best possible result for \eqref{duhamel}. To see this, set $$\widehat{u}=\chi_{Q_1},\widehat{v}=\chi_{Q_2},\widehat{w}=\chi_{Q_2} $$
where $\chi_S$ is the characteristic function of the set $S$ and
\begin{align*}
&Q_1=\{(\xi_1,\xi_2,\tau)\in \mathbb{R}^3: |\xi_1|\leq 1, |\xi_2-N|\leq 1/N, |\tau+N^2|\leq 1  \} \\& Q_2=\{(\xi_1,\xi_2,\tau)\in \mathbb{R}^3: |\xi_1|\leq 1 , |\xi_2|\leq 1/\sqrt{N}, |\tau|\leq 1  \}
\end{align*}
for $N\in \mathbb{N}$ large. Since the volume of the former set $\sim 1/N$ and of the latter $\sim N^{-1/2}$, we have $\norm{u}_{X^{s,b}}\sim N^{s-1/2}$ and $\norm{v}_{X^{s,b}}=\norm{w}_{X^{s,b}}\sim N^{-1/4}$. Note that
\begin{align*}\widehat{u\overline{v}w}(x,t)&=\int_{\mathbb{R}^3}\int_{\mathbb{R}^3}\chi_{Q_1}(x-\eta-\zeta,t-\lambda-\tau)\chi_{Q_2}(\eta,\lambda)\chi_{Q_2}({\zeta,\tau})\,\text{d}\eta\, \text{d}\lambda\, \text{d}\zeta\, \text{d}\tau \\&=\int_{Q_2}\int_{Q_2}\chi_{Q_1}(x-\eta-\zeta,t-\lambda-\tau)\,\text{d}\eta\, \text{d}\lambda\, \text{d}\zeta\, \text{d}\tau
\end{align*}
and
\begin{align*}
&(x-\eta-\zeta,t-\lambda-\tau)\in Q_1\Rightarrow (\zeta+\eta,\tau+\lambda)\in (x,t)-Q_1 
\\&(\zeta,\tau), (\eta,\lambda)\in Q_2\Rightarrow(\zeta+\eta,\tau+\lambda)\in 2Q_2, 
\end{align*}
where we have used the fact that $Q_2+Q_2=2Q_2$ as $Q_2$, having box like shape, is convex. Hence $(\zeta+\eta,\tau+\lambda)\in 2Q_2\cap [(x,t)-Q_1]$ (notice $2Q_2=\{(\xi_1,\xi_2,\tau)\in \mathbb{R}^3: |\xi_1|\leq 2,|\xi_2|\leq 2N^{-1/2}, |\tau| \leq 2\}$). Therefore in order that $2Q_2\cap [(x,t)-Q_1]\neq \emptyset$, we should have $|x_1-\xi_1|\leq 2, |x_2-\xi_2|\leq 2N^{-1/2}, |t-\tau|\leq 2$ for $(\xi_1,\xi_2,\tau)\in Q_1$ and $x=(x_1,x_2)$. Hence using these relations and taking $N$ sufficiently large, $(x,t)$ can be thought to belong to a set that covers $Q_1$. Therefore for $(x,t)\in Q_1$, we can find some $C>0$ such that  
 \begin{align*}
 |\widehat{u\overline{v}w}(x,t)|\geq CN^{-1} \chi_{Q_1}(x,t).
 \end{align*}
 This follows
 \begin{align*}
 ||u\overline{v}w||_{X^{s+a,b-1}}& \geq CN^{-1}\Big(\int_{Q_1}\langle\tau+|\xi|^2\rangle^{2(b-1)}\langle\xi\rangle^{2(s+a)}\,\text{d}\xi\, \text{d}\tau\Big)^{1/2}\\&\sim CN^{s+a-1}\Big(\int_{Q_1}\text{d}\xi\, \text{d}\tau\Big)^{1/2}\geq CN^{s+a-3/2}.
 \end{align*}
 Thus for $a>1/2$ Proposition \ref{key} implies that $CN^{s+a-3/2}\lesssim N^{s-1/2}N^{-1/4}N^{-1/4}=N^{s-1}\Rightarrow C\lesssim N^{1/2-a}$;
 nevertheless as $N$ gets larger, the right hand side of this inequality becomes arbitrarily small, while the left side is a fixed positive number.  \\ In the sequel we make use of the ideas in \cite{1} and \cite{Talbot} so as to finish the proof of Theorem \ref{mainthm}.

 \section{Proof of Theorem \ref{mainthm}} \label{pfmainthm}
Let $\delta$ be the local existence time given by the local theory. Since $b>\frac{1}{2}$ we use the embedding  $X^{s+a,b}\hookrightarrow {C^0_tH_x^{s+a}}$ along with Lemma \ref{delta}, Proposition \ref{key} and local theory bound to obtain 
\begin{multline}\label{smoothinglocal}
\norm{u(t)-e^{it\Delta}u(0)}_{C^0_tH_x^{s+a}([-\delta,\delta]\times \mathbb{R}^2)} \lesssim \norm{\int_0^t e^{i(t-\tau)\Delta}[c_1|u|^2u+c_2K(|u|^2)u](\tau)d\tau}_{X^{s+a,b}_{\delta}}\\ \lesssim \norm{c_1|u|^2u+c_2K(|u|^2)u}_{X^{s+a,b-1}_{\delta}}\lesssim \norm{u}_{X^{s,b}_{\delta}}^3\lesssim \norm{u_0}_{H^s}^3.
\end{multline}
Therefore \eqref{smoothinglocal} proves \eqref{smoothing} in the local time interval $[-\delta,\delta]$, in order to prove it for all times we have to iterate this result and then prove the continuity argument. To do so fix $t$ large then for if $r\leq t$, from \eqref{T(t)} with the choice of $\beta(s)$ (remember that for $s>1/2$ this choice makes $T(t)$ non-decreasing function of time) we have    
\begin{align*}
\norm{u(r)}_{H^s}\lesssim T(r)\leq T(t).
\end{align*}
In particular for any $j$ with $j\delta\leq t$ we have the bound \begin{align*}
\norm{u((j-1)\delta)}_{H^s}\lesssim T((j-1)\delta)\leq T(t).
\end{align*}
Thus, under favour of global well-posedness result of Theorem \ref{global}, considering the initial value problem \eqref{ee} with $u((j-1)\delta)$ being the initial data, and implementing \eqref{smoothinglocal} to this local problem, we obtain \begin{align*}
\norm{u(j\delta)-e^{i\delta\Delta}u((j-1)\delta)}_{H^{s+a}}\lesssim \norm{u((j-1)\delta)}_{H^s}^3\lesssim  T(t)^3.
\end{align*}
By the local theory we pick $\delta\sim T(t)^{-\frac{2}{s}}$ so that, for $J=t/\delta\sim tT(t)^{\frac{2}{s}}$, we get
\begin{align*}
\norm{u(t)-e^{it\Delta}u_0}_{H^{s+a}}&=\norm{u(J\delta)-e^{i\delta J\Delta}u(0)}_{H^{s+a}}\\& =\norm{\sum_{j=1}^{J} e^{i\delta(J-j)\Delta} u(j\delta)-e^{i\delta (J-j+1)\Delta}u((j-1)\delta)  }_{H^{s+a}} \\& \leq \sum_{j=1}^{J}\norm{ e^{i\delta(J-j)\Delta} u(j\delta)-e^{i\delta (J-j+1)\Delta}u((j-1)\delta)  }_{H^{s+a}} \\& =\sum_{j=1}^{J}\norm{  u(j\delta)-e^{i\delta \Delta}u((j-1)\delta)  }_{H^{s+a}} \\&\lesssim JT(t)^3 \lesssim \langle t \rangle T(t)^{3+\frac{2}{s}}.
\end{align*}
This finishes the iteration argument, in order for \eqref{smoothing} to hold we are left to show that the difference
  \begin{align*}
  D(t):=u(t)-e^{it\Delta}u_0=-i\int_0^te^{i(t-\tau)\Delta}(c_1|u|^2u+c_2K(|u|^2)u)(\tau)d\tau
  \end{align*}
 is continuous in $H^{s+a}$. Assume that $r$ is fixed and  also without loss of generality that $t>r$, then
  \begin{multline*}
  D(t)-D(r)=i\int_0^re^{i(r-\tau)\Delta}(c_1|u|^2u+c_2K(|u|^2)u)(\tau)d\tau \\ -i\int_0^te^{i(t-\tau)\Delta}(c_1|u|^2u+c_2K(|u|^2)u)(\tau)d\tau \\ =-i(e^{i(t-r)\Delta}-\text{Id})\int_0^r e^{i(r-\tau)\Delta}(c_1|u|^2u+c_2K(|u|^2)u)(\tau)d\tau \\ -i\int_r^t e^{i(t-\tau)\Delta}(c_1|u|^2u+c_2K(|u|^2)u)(\tau)d\tau.
  \end{multline*}  Thus 
  \begin{multline*}
  \norm{D(t)-D(r)}_{H^{s+a}}\\
  \lesssim \norm{\langle \xi \rangle^{s+a} \mathcal{F}\Big((e^{i(t-r)\Delta}-\text{Id})\int_0^r e^{i(r-\tau)\Delta}(c_1|u|^2u+c_2K(|u|^2)u)(\tau)d\tau\Big)}_{L^2_{\xi}} \\ \hspace{2.3cm}+ \norm{\langle \xi \rangle^{s+a} \mathcal{F}\Big(\int_r^t e^{i(t-\tau)\Delta}(c_1|u|^2u+c_2K(|u|^2)u)(\tau)d\tau\Big)}_{L^2_{\xi}} \\=: \text{I}_1+ \text{I}_2.
  \end{multline*}
 Using the inequality $|e^{i(r-t)|\xi|^2}-1|\lesssim \min(1,|\xi|^2|t-r|)\leq (|\xi|^2|t-r|)^{\epsilon}$ in the subsequent calculation (for a sufficiently small $\epsilon>0$), we obtain
  \begin{multline}\label{39}
  \text{I}_1 =\Bigl|\!\Bigl|\langle \xi \rangle^{s+a} \mathcal{F}\Big( \mathcal{F}^{-1}\Big(e^{-i(t-r)|\xi|^2}\mathcal{F}\Big(\int_0^r e^{i(r-\tau)\Delta}(c_1|u|^2u+c_2K(|u|^2)u)(\tau)d\tau\Big)\Big)\\-\int_0^r e^{i(r-\tau)\Delta}(c_1|u|^2u+c_2K(|u|^2)u)(\tau)d\tau\Big)\Bigl|\!\Bigl|_{L^2_{\xi}} \\ =  \norm{\langle \xi \rangle^{s+a} (e^{-i(t-r)|\xi|^2}-1)\mathcal{F}\Big(\int_0^r e^{i(r-\tau)\Delta}(c_1|u|^2u+c_2K(|u|^2)u)(\tau)d\tau\Big)}_{L^2_{\xi}}\\ \lesssim |t-r|^{\epsilon}\norm{\langle \xi \rangle^{s+a+2\epsilon}\mathcal{F}\Big(\int_0^r e^{i(r-\tau)\Delta}(c_1|u|^2u+c_2K(|u|^2)u)(\tau)d\tau\Big)}_{L^2_{\xi}} \\ =|t-r|^{\epsilon}\norm{\int_0^r e^{i(r-\tau)\Delta}(c_1|u|^2u+c_2K(|u|^2)u)(\tau)d\tau}_{H^{s+2\epsilon+a}}\\ = |t-r|^{\epsilon}\norm{\sum_{j=1}^{r/\delta}\int_{(j-1)\delta}^{j\delta} e^{i(r-\tau)\Delta}(c_1|u|^2u+c_2K(|u|^2)u)(\tau)d\tau}_{H^{s+2\epsilon+a}}\\ \leq|t-r|^{\epsilon} \sum_{j=1}^{r/\delta}\norm{\int_{(j-1)\delta}^{j\delta} e^{i(r-\tau)\Delta}(c_1|u|^2u+c_2K(|u|^2)u)(\tau)d\tau}_{H^{s+2\epsilon+a}} 
  \end{multline}
 where we pick the same $\delta$ given by the local existence time of the solution. The sum above contains no more than finitely many terms because $\delta$ depends upon $\sup\limits_{t\in[0,r]}\norm{u(t)}_{H^s(\mathbb{R}^2)}$ that is finite by Theorem \ref{global} and $r$ is fixed. As the length of each interval of integration is $\delta$, by the time translation and time reversal symmetries of the solution, it suffices just to estimate the following integral for $t\in[-\delta/2,\delta/2]$:
 \begin{multline*}
  \norm{\int_0^{t} e^{i(t-\tau)\Delta}(c_1|u|^2u+c_2K(|u|^2)u)(\tau)d\tau}_{H^{s+2\epsilon+a}}\\\leq  \norm{\eta(t/\delta)\int_0^{t} e^{i(t-\tau)\Delta}(c_1|u|^2u+c_2K(|u|^2)u)(\tau)d\tau}_{L^{\infty}_tH^{s+2\epsilon+a}_{x,y}} 
 \end{multline*}
 where $\eta(t)$ is a $C^{\infty}$ function supported on $[-2,2]$ that equals to $1$ on $[-1,1]$. Next we use the embedding, $X^{s+2\epsilon+a,\tilde{b}}\subset {L^{\infty}_tH^{s+2\epsilon+a}_{x,y}}$ for $\tilde{b}>\frac{1}{2}$ (set for instance $\tilde{b}=1/2+\epsilon$ for $\epsilon>0$), to majorize the integral above by \begin{align*}
 \norm{\eta(t/\delta)\int_0^{t} e^{i(t-\tau)\Delta}(c_1|u|^2u+c_2K(|u|^2)u)(\tau)d\tau}_{X^{s+2\epsilon+a,\tilde{b}}}
 \end{align*}
  which, by means of Lemma \ref{delta} with $b'=-\frac{1}{2}+2\epsilon$ and $b'=b-1$, can be estimated by
  \begin{align*}
 \delta^{\epsilon} \norm{ c_1|u|^2u+c_2K(|u|^2)u}_{X^{s+2\epsilon+a,b-1}_{\delta}}
 \end{align*}
 Finally an application of Proposition \ref{key} with $s+2\epsilon$ and the local theory gives the bound
  \begin{align*}
 \delta^{\epsilon} \norm{u}_{X^{s+2\epsilon,b}_{\delta}}^3 \lesssim \delta^{\epsilon}\norm{u_0}_{H^{s+2\epsilon}}^3.
 \end{align*}
 As a result, the sum in \eqref{39} is bounded and hence $I_1$ converges to $0$ as $t\rightarrow r$. Also the same result follows for $\text{I}_2$ by using the Proposition \ref{key} and the Lemma \ref{delta}.
 
 \section{existence of a global attractor: proof of thm
 \ref{thm:attractor}}\label{attractor}
 This section is devoted to give an alternative proof of the existence of a global attractor for the forced and weakly damped DS by using smoothing estimates. Firstly we show that the evolution operator is weakly continuous then we exploit this in handling the corresponding energy equation to upgrade the weak convergence, resulting from boundedness of the flow, to strong convergence giving rise to the asymptotic compactness of the flow. Throughout this section, $\xrightarrow{w}$ and $\xrightarrow{w^*}$ will denote the weak and the weak$^*$ convergences respectively.  
\begin{lemma}\label{weakcont}
If $u_0^k\xrightarrow{w}u_0$ in $H^1$, then for all $T>0$ and $a\in(0,\frac{1}{2})$, the linear and the nonlinear parts of the semigroup operator $U(t)$ of \eqref{elelforced} satisfy 
\begin{align*}L^{\delta}(t)u_0^k\xrightarrow{w}L^{\delta}(t)u_0\,\,\text{in}\,\,L^2([0,T];H^1)\\
N(t)u_0^k\xrightarrow{w}N(t)u_0\,\,\text{in}\,\,L^2([0,T];H^{1+a}).
\end{align*}
Moreover, for $t\in[0,T]$,
\begin{align*}L^{\delta}(t)u_0^k\xrightarrow{w}L^{\delta}(t)u_0\,\,\text{in}\,\,H^1\\
N(t)u_0^k\xrightarrow{w}N(t)u_0\,\,\text{in}\,\,H^{1+a}.
\end{align*}

\end{lemma}
\begin{proof}
We just verify the assertions concerning the nonlinear part as the ones for the linear part will follow from the Fourier representation of the linear flow at once. To make use of the smoothing result for the forced problem \eqref{elelforced}, we shall transform the equation \eqref{elelforced} with data $u_0$ by setting $$g=\big(\widehat{f}/\langle\xi\rangle^2\big)^{\vee}=(1-\Delta)^{-1}f\in H^2,$$ 
and $v=u+g$ so that obtain the equation 
\begin{align}\label{transformed}
    iv_t+\Delta v+i\delta v=c_1|v-g|^2(v-g)+c_2(v-g)\phi_x +(1+i\delta g),
\end{align}
with $\phi_x=K(|v-g|^2)$ and data $v(\cdot,0)=u_0(\cdot)+g(\cdot)$. Let $L^{\delta}$ denote the semigroup for the corresponding homogeneous linear equation \begin{align}\label{linearw}
iw_t+\Delta w+i\delta w=0.
\end{align} Then the nonlinear part $n=v-w$ satisfies the equation
\begin{align}\label{noneqn}
in_t+\Delta n+i\delta n=c_1|n+w-g|^2(n+w-g)+c_2(n+w-g)K(|n+w-g|^2)+(1+i\delta g)
\end{align}
with data $n(\cdot,0)=0$. Proceeding as in Section \ref{pfmainthm}, notably using the trilinear estimate \eqref{keyest} with $s=1$, $a\in (0,\frac{1}{2})$, for $T\leq 1$, we get 
\begin{align} \label{diss.local}
\norm{n}_{C^0_tH^{1+a}([0,T]\times \mathbb{R}^2)}\lesssim \norm{u_0}_{H^1}^3+\norm{g}_{H^{a-1}}\leq \norm{u_0}_{H^1}^3+\norm{f}_{L^{2}}.
\end{align}
Note that in the above estimate we use a variant of Lemma \ref{delta} that replaces the linear group for the nondissipative equation by the dissipative group $L^{\delta}(t)=e^{it\Delta-t\delta}$. For a proof of this, see \cite{Erdoganlong}. Using the $H^1$ global wellposedness (a priori bound \eqref{globalH1}) in iterating the local result above, we obtain the global bound \begin{align} \label{globalbound}
\norm{n(t)}_{H^{1+a}}\leq C(a,\delta, \norm{f}_{L^2} \norm{u_0}_{H^1}),  \end{align}
for the details see the section $3$ of \cite{Erdoganlong}, or section $6$ of \cite{ErdoganZak}. Lastly the continuity in $H^{1+a}$ follows as in Section \ref{pfmainthm}, so for any $T>0$ and initial data $u_0\in H^1$, we have
\begin{align}\label{diss.global}
 n \in C^0_tH^{1+a}([0,T]\times \mathbb{R}^2) .  
\end{align}  
In order to show that $N(t)u_0^k\xrightarrow{w}N(t)u_0$ in spaces given by the statement of the lemma, it suffices to show that every subsequence of $N(t)u_0^k$ has a further subsequence which converges weakly to the same limit. Let $w^k$ denote the solution to \eqref{linearw} with data $u_0^k$. In relation to this denote by $n^k$ the nonlinear part. Since weak convergence $u_0^k\xrightarrow{w}u_0$ in $H^1$ implies that $\sup_k\norm{u_0^k}_{H^1}\leq M$ for some $M>0$, using this in \eqref{diss.local} and \eqref{globalbound}, we infer that, for every $T>0$, $\{n^k\}_k$ is bounded in 
 \begin{align}\label{boundedness}
 C([0,T];H^{1+a})\cap C^1([0,T];H^{a-1})
\end{align}
with a uniform bound of \eqref{globalbound}.
Firstly, in conjuction with this boundedness we infer, by Arzel\`{a}-Ascoli theorem, that $\{n^k\}_k$ is relatively compact in $C([0,T];H_{loc}^{a-1})$ thanks to the uniform boundedness of the derivatives which implies equicontinuity. Therefore interpolating between this and \eqref{boundedness} gives us a subsequence of $\{n^k\}_k$ that converges strongly in $C([0,T];H_{loc}^{1+a})$. Secondly, by the Banach-Alaoglu theorem, boundedness of $\{n^k\}_k $ in \eqref{boundedness} yield a weak* convergent subsequence in $L^{\infty}([0,T];H^{1+a})$. Therefore combining these two, we reach a further subsequence, denoted also by $\{n^k\}_k$, convergent in the above spaces with the corresponding type of convergences. We shall write $n^k\xrightarrow{w^*}n$ in $L^{\infty}([0,T];H^{1+a})$, and $n^k\rightarrow n$ strongly in $C([0,T];H_{loc}^{1+a})$ (weak$^*$ limits are unique). The analogous arguments for the linear parts $w^k$ hold in $H^1$ as well, so denote the corresponding limit by $w$. Later we will see that the limit $n$ is indeed the weak limit in the spaces dictated by the lemma. Using local strong convergence above, next, we will show that $n$ is a distributional solution. Note that $n^k$ satisfies the equation \eqref{noneqn} and let $F(p,q)=c_1|p+q-g|^2(p+q-g)+c_2(p+q-g)K(|p+q-g|^2)$. Thus for any $\varphi \in C_c^{\infty}([0,T];\mathbb{R}^2)$, 
\begin{multline*}
 \iint\Big(in_t+\Delta n+i\delta n-F(n,w)-(1+i\delta )g\Big)\varphi\, \text{d}x\,\text{d}t\\= \iint \Big(\big[-i\varphi_t+\Delta \varphi+i\delta\varphi\big]n-\big[(1+i\delta g)+ F(n,w)\big]\varphi\Big)\, \text{d}x\,\text{d}t \\= \lim_{k\rightarrow \infty}\iint \Big(\big[-i\varphi_t+\Delta \varphi+i\delta\varphi\big]n^k-\big[(1+i\delta g)+ F(n^k,w^k)\big]\varphi\Big)\, \text{d}x\,\text{d}t\\=\lim_{k\rightarrow \infty}\iint\Big(in^k_t+\Delta n^k+i\delta n^k-F(n^k,w^k)-(1+i\delta )g\Big)\varphi\, \text{d}x\,\text{d}t=0,
\end{multline*}
which proves that $n$ is a distributional solution. It is just left to verify the second equality above. It suffices to show that the following identity \begin{multline*}
\Big|\iint \Big(\big[-i\varphi_t+\Delta \varphi+i\delta\varphi\big](n^k-n)-\big[F(n^k,w^k)-F(n,w)\big]\varphi\Big)\, \text{d}x\,\text{d}t   \Big|\\ \leq \Big|\iint \big[-i\varphi_t+\Delta \varphi+i\delta\varphi\big](n^k-n)\text{d}x\,\text{d}t \Big|+\norm{\varphi}_{L^{\infty}_{x,t}}\Big|\iint\limits_{\text{supp}\varphi}\big[F(n^k,w^k)-F(n,w)\big]\, \text{d}x\,\text{d}t   \Big|
\end{multline*}
eventually decreases to zero for increasing $k$.
The first integral is clearly decaying due to strong local convergence of $n^k$, whereas the integral part of the second summand is majorized by 
\begin{multline*}
 \Big|\iint\limits_{\text{supp}\varphi}\big(n^k-n+w^k-w\big)\big[c_1I+c_2K\big](|n^k+w^k-g|^2)\, \text{d}x\,\text{d}t\Big|\\+ \Big|\iint\limits_{\text{supp}\varphi}(\overline{n^k-n+w^k-w})(n^k+w^k-g)\big[c_1I+c_2K\big](n+w-g)\, \text{d}x\,\text{d}t\Big|\\+\Big|\iint\limits_{\text{supp}\varphi}\big(n^k-n+w^k-w\big)\big(\overline{n+w-g}\big)\big[c_1I+c_2K\big](n+w-g)\, \text{d}x\,\text{d}t\Big|  
\end{multline*}
where we have used the property $(\text{iii})$ of the operator $K$ given in Section \ref{prelim} in the computation above. Owing to the strong local convergences of $n^k$ and $w^k$ along with the boundedness of $K$, we conclude that these sums vanish in the limit. Moreover, by the uniqueness of the weak$^*$ limits, $n$ is a unique distribution belonging to $C([0,T];H^{1+a})$ yielding that $n=N(t)u_0$ (since $n$ has shown to be the distributional solution of \eqref{noneqn}). Therefore, using the fact that $L^2([0,T];H^{1+a})$ embeds in the dual of $L^{\infty}([0,T];H^{1+a})$, weak$^*$ convergence in $L^{\infty}([0,T];H^{1+a})$ implies the weak convergence in $L^2([0,T];H^{1+a})$. This finishes the proof of the first assertion in the lemma. 

To prove the second argument, fix a $\tilde{t}\in [0,T]$. As before smoothing estimate together with the boundedness of the initial data $u^k_0$ imply that $\{N(\tilde{t})u_0^k\}_k$ is bounded in $H^{1+a}$. Thus there exists a weakly convergent subsequence, still denoted by $\{N(\tilde{t})u_0^k\}$, that converges in $H^{1+a}$, say to $\widetilde{n}$. But as we know, from the previous discussion above, that $N(t)u_0^k\xrightarrow{w^*} N(t)u_0$ in $C([0,T];H^{1+a})$. So the uniqueness entails that $\widetilde{n}=N(\tilde{t})u_0$.

\end{proof}

\begin{proof}[Proof of Theorem~\ref{thm:attractor}]
To begin with, we note that the existence of an absorbing ball $\mathcal{B}$ of the evolution follows from \eqref{globalH1}, indeed the detailed proof was given in \cite{Wang}. Hence to attain a global attractor, it is just left to affirm that the propogator $U(t)$ is asymptotically compact. That is, it is sufficient to show that for any sequence of initial data $\{u_{0,k}\}_k$ in absorbing ball and any sequence of times $t_k\rightarrow \infty$, the sequence $\{U(t_k)u_{0,k}\}_k$ has a convergent subsequence in $H^1$. Note that for $u_0\in\mathcal{B}$, \eqref{globalbound} implies that the nonlinear part $N(t)u_0$ of \begin{align*}
U(t)u_0=L^{\delta}(t)u_0+N(t)u_0
\end{align*} is contained in a ball $B_R$ in $H^{1+a}$ with radius $R=R(a,\delta,\norm{f}_{L^2})$, $a\in(0,\frac{1}{2})$. As a result, $\{N(t_k)u_{0,k}\}_k\subset B_R$. Therefore we can find a subsequence, still denoted by $N(t_k)u_{0,k}$, that converges weakly in $H^{1+a}$. Moreover since the weak and weak$^*$ topologies agree on a reflexive spaces, the Banach-Alaoglu theorem yields that, up to a subsequence, $U(t_k)u_{0,k}$ converges weakly in $H^1$. As $L^{\delta}(t_k)u_{0,k}\rightarrow 0$ strongly in $H^1$ as $t_k\rightarrow \infty$, $N(t_k)u_{0,k}$ and $U(t_k)u_{0,k}$ converge to the same limit, say to u. Furthermore, for every $T>0$, we can find a further subsequence so that $N(t_k-T)u_{0,k}$ and $U(t_k-T)u_{0,k}$ converge weakly in $H^{1+a}$ and $H^1$ respectively. As above, by the decay of the linear part, the limits are the same, so denote it by $u_T$. By Lemma \ref{weakcont},
\begin{align*}
U(t_k-T)u_{0,k}\xrightarrow{w} u_T\,\,\text{in}\,\,H^1\implies U(t_k)u_{0,k}=U(T)\big(U(t_k-T)u_{0,k}\big)\xrightarrow{w}U(T)u_T \,\,\text{in}\,\,H^1.  
\end{align*}
Therefore, by the uniqueness of a weak limit, $U(T)u_T=u$. In a subsequent discussion, sometimes we need to take $T\rightarrow \infty$ in order to obtain strong convergences. So to make sense of this, we may implement a diagonalization argument for a countable set $\{T\in \mathbb{N}\}$ so that, up to a same subsequence for all $T$, $U(t_k-\cdot)u_{0,k}$ and $N(t_k-\cdot)u_{0,k}$ converge weakly at each $T$ in the corresponding spaces above.   

Next we want to upgrade the weak $H^1$ convergence of the solution flow $U(t_k)u_{0,k}$ to a strong $H^1$ convergence. Firstly using the equation \eqref{forcedsingle}, we can obtain $$\frac{d}{dt}\norm{u}_{L^2}+2\,\delta \norm{u}_{L^2}=2\,\text{Im}\langle f,u \rangle,$$ and then application of the Gronwall lemma for the evolution $U(t)$ gives that
\begin{align*}
\norm{U(t_k)u_{0,k}}_{L^2}^2&=e^{-2\delta T}\norm{U(t_k-T)u_{0,k}}_{L^2}^2+2\,\text{Im}\int_0^T e^{-2\delta (T-\tau)}\langle f,U(t_k-T+\tau)u_{0,k} \rangle_{L^2_x}\, \text{d}\,\tau \\&  \norm{U(T)u_T}_{L^2}^2=e^{-2\delta T}\norm{u_T}_{L^2}^2+2\,\text{Im}\int_0^T e^{-2\delta (T-\tau)}\langle f,U(\tau)u_T \rangle_{L^2_x}\, \text{d}\,\tau,
\end{align*}
that yield
\begin{align*}
 \norm{U(t_k)u_{0,k}}_{L^2}^2- \norm{U(T)u_T}_{L^2}^2&= e^{-2\delta T}\big(\norm{U(t_k-T)u_{0,k}}_{L^2}^2-\norm{u_T}_{L^2}^2\big)\\&+2\,\text{Im}\int_0^T e^{-2\delta (T-\tau)}\langle f,U(t_k-T+\tau)u_{0,k}-U(\tau)u_T \rangle_{L^2_x}\, \text{d}\,\tau.
\end{align*}
The first summand becomes negligible by taking sufficiently large $T$ since letting $k\rightarrow \infty$ and using the fact that $u_{0,k}\in \mathcal{B}$, we ensure that the norms in the parantheses are finite (weak convergence in $H^1$ implies that $\norm{u_T}_{L^2}\leq \lim\limits_k\inf \norm{U(t_k-T)u_{0,k}}_{L^2}$). The second summand vanishes in the limit by Lemma \ref{weakcont} because, with $U(T)u_T=u$, we have \begin{multline*}
U(t_k)u_{0,k}\xrightarrow{w}u \,\,\text{in}\,\,H^1\implies \\ U(t_k-T+\tau)u_{0,k}=U(\tau-T)\big(U(t_k)u_{0,k}\big)\xrightarrow{w}U(\tau)u_T\,\,\text{in}\,\,L^2([0,T];H^1).
\end{multline*}
As a consequence we get that
\begin{align*}
\lim\limits_k\sup \big(\norm{U(t_k)u_{0,k}}_{L^2}^2-\norm{U(T)u_T}_{L^2}^2\big)=\lim\limits_k\sup \big(\norm{U(t_k)u_{0,k}}_{L^2}^2-\norm{u}_{L^2}^2\big)\leq 0,    
\end{align*}
which, along with $U(t_k)u_{0,k}\xrightarrow{w}u$ in $H^1$, implies that $U(t_k)u_{0,k}\rightarrow u$ strongly in $L^2$. This strong $L^2$ convergence will be important in the upcoming energy calculations. So define the functional $E$ by \begin{multline*}E(u_0)(t)=\norm{\nabla U(t)u_0}_{L^2}^2+\frac{c_1}{2}\norm{U(t)u_0}_{L^4}^4+\frac{c_2}{2}\int K(|U(t)u_0|^2)|U(t)u_0|^2\,\text{d}x\\+ 2\,\text{Re}\,\int f\overline{U(t)u_0}\,\text{d}x,\end{multline*}
and the time derivative is as follows
$$\frac{d}{dt}E(u_0)(t)=-2\delta E(u_0)(t)+F(u_0)(t)$$
where
$$F(u_0)(t)=-\delta c_1 \norm{U(t)u_0}_{L^4}^4-\delta c_2\int K(|U(t)u_0|^2)|U(t)u_0|^2\,\text{d}x+2\delta \,\text{Re}\,\int f\overline{U(t)u_0}\,\text{d}x.$$
Gronwall lemma implies that
$$E(u_{0,k})(t_k)-E(u_T)(T)=\sum\limits_{j=1}^4I_j,$$
where
\begin{align*}
I_1&=e^{-2\delta T}\big(E(u_{0,k})(t_k-T)-E(u_T)(0)\big)
\\ I_2&=-\delta c_1 \int_0^T e^{-2\delta (T-\tau)}\big(\norm{U(t_k-T+\tau)u_{0,k}}_{L^4}^4-\norm{U(\tau)u_T}_{L^4}^4\big)\text{d}\tau
\\I_3&=-\delta c_2\int_0^T\int e^{-2\delta (T-\tau)}\big(K(|U(t_k-T+\tau)u_{0,k}|^2)|U(t_k-T+\tau)u_{0,k}|^2\\& \hspace{7.2cm} -K(|U(\tau)u_T|^2)|U(\tau)u_T|^2\big)\text{d}x\,\text{d}\tau\\
I_4&=2\delta\, \text{Re}\int_0^Te^{-2\delta (T-\tau)}\langle f,U(t_k-T+\tau)u_{0,k}-U(\tau)u_{0,k}\rangle_{L^2_x}\text{d}\tau. \end{align*}
$I_1$ gets arbitrarily small by increasing $T$, and $I_2$ can be majorized by
\begin{align*}
 &\delta c_1\int _0^Te^{-2\delta (T-\tau)}\norm{U(t_k-T+\tau)u_{0,k}-U(\tau)u_T}_{L^4}\\& \hspace{4cm}\times\big(\norm{U(t_k-T+\tau)u_{0,k}}_{L^4}+\norm{U(\tau)u_T}_{L^4}\big)^3\text{d}\tau, 
\end{align*}
note that by $L^4$ Gagliardo-Nirenberg inequality 
 \begin{multline*}
\norm{U(t_k-T+\tau)u_{0,k}-U(\tau)u_T}_{L^4}\\\lesssim\norm{U(\tau-T)\big(U(t_k)u_{0,k}-U(T)u_T\big)}_{L^2}^{1/2}\norm{U(t_k-T+\tau)u_{0,k}-U(\tau)u_T}_{H^1}^{1/2}.
\end{multline*}
Then by the strong continuity of $U(\tau-T)$ and the strong $L^2$ convergence $U(t_k)u_{0,k}\rightarrow U(T)u_T$, the majorant of $I_2$ above vanishes in the limit as weak $H^1$ convergence yields the $H^1$ boundedness of the norms. Write the third term as
\begin{multline*}
 I_3=-\delta c_2\int_0^T\int e^{-2\delta(T-\tau)}\big[K(|U(t_k-T+\tau)u_{0,k}|^2)-K(|U(\tau)u_T|^2)\big ]|U(t_k-T+\tau)u_{0,k}|^2\text{d}x\,\text{d}\tau \\   -\delta c_2\int_0^T\int e^{-2\delta(T-\tau)}K(|U(\tau)u_T|^2)\big(|U(t_k-T+\tau)u_{0,k}|^2-|U(\tau)u_T|^2\big)\text{d}x\,\text{d}\tau,
\end{multline*}
first use the linearity and the $L^p$ boundedness of the operator $K$, then proceed by using the same reasoning as above to conclude that $I_3$ decays to zero. Finally, by using the weak continuity (Lemma \ref{weakcont}), $I_4$ vanishes in the limit. Therefore, as $U(T)u_T=u$, we get that
$$\lim\limits_{k\rightarrow \infty}\sup \big(E(u_{0,k})(t_k)-E(u)(0)\big)=\lim\limits_{k\rightarrow \infty}\sup \big(E(u_{0,k})(t_k)-E(u_T)(T)\big)\leq 0.$$
This inequality together with the definition of $E$ and taking limits as above imply that $$\lim\limits_{k\rightarrow\infty}\sup \norm{\nabla U(t_k)u_{0,k}}_{L^2}^2\leq \norm{\nabla u}_{L^2}^2.$$ Therefore this with $U(t_k)u_{0,k}\xrightarrow{w}u$ in $H^1$ lead to the $L^2$ strong convergence of $\nabla U(t_k)u_{0,k}$ to $\nabla u$. Consequently $U(t_k)u_{0,k}\rightarrow u$ strongly in $H^1$. This finishes the proof of asymptotic compactness, so the proof is complete. 
\end{proof}
\section*{Acknowledgement}
The author would like to thank his Ph.D advisor T. Burak G\"{u}rel for his guidance and many salutary discussions.


\begin{thebibliography}{22}
 
\bibitem{Benney}
D. J. Benney, G. L. Roskes, \textit{Wave instabilities}, Stud. Appl. Math. {\bf 48} (1969), 377--385.

\bibitem{Colliander}
J. Colliander, M. Keel, G. Staffilani, H. Takaoka, T. Tao, \textit{Almost conservation laws and global
rough solutions to a nonlinear Schr\"odinger equation}, Mathematical Research Letters (2002), vol. 9, no. 5-6, 659--682.
\bibitem{Compaan}
 E. Compaan, \textit{Smoothing for the Zakharov and the Klein-Gordon-Schr\"{o}dinger systems on Euclidean spaces}, SIAM J. Math. Anal. {\bf 49} (2017), 4206--4231.
\bibitem{Davey}
A. Davey, K. Stewartson, \textit{On three-dimensional packets of surface waves}, Proc. R. Soc. A. {\bf 338} (1974), 101--110.
\bibitem{Djordjevic}
V. D. Djordjevic, L. G. Redekopp, \textit{On two-dimendional packets of capillary-gravity waves}, J.
Fluid Mech. {\bf 79} (1977), 703--714.
\bibitem{1}
M. B. Erdo\u{g}an, N. Tzirakis, \textit{Global smoothing for the periodic KdV evolution}, Int. Math. Res. Not. (2013), no. 20, 4589--4614.
\bibitem{Erdoganlong}
M. B. Erdo\u{g}an, N. Tzirakis, \textit{Long time dynamics for the forced and weakly damped KdV on the torus}, Commun. Pure Appl. Anal. {\bf 12} (2013), no. 6, 2669--2684.
\bibitem{Talbot}
M. B. Erdo\u{g}an, N. Tzirakis, \textit{Talbot effect for the cubic nonlinear Schr\"odinger equation on the
torus}, Math. Res. Lett. {\bf 20} (2013), 1081--1090.
\bibitem{ErdoganZak}
M. B. Erdo\u{g}an, N. Tzirakis, \textit{Smoothing and global attractors for the Zakharov system on the torus}, Anal. PDE {\bf 6} (2013), no. 3, 723--750.

\bibitem{Fokas}
A. S. Fokas, \textit{The Davey-Stewartson Equation on the Half-Plane}, Commun. Math. Phys. {\bf 289} (2009), 957--993.

\bibitem{Ghidaglia}
 J. M. Ghidaglia, J. C. Saut, \textit{On the initial value problem for the Davey-Stewartson systems}, Nonlinearity {\bf 3} (1990), no. 2, 475--506.
\bibitem{Ginibre}
J. Ginibre, Y. Tsutsumi, G. Velo, \textit{On the Cauchy problem for the Zakharov system}, J. Functional
Analysis {\bf 151} (1997), 384--436.
\bibitem{Godet}
N. Godet, \textit{A lower bound on the blow-up rate for the Davey Stewartson system on the torus,} Ann. I. H. Poincar\'e  AN 30 (2013), 691--703.
\bibitem{2}
B. L. Guo, B. X. Wang, \textit{The Cauchy problem for Davey-Stewartson systems}, Communications on
Pure and Applied Mathematics {\bf 52} (1999), no. 12, 1477--1490.
\bibitem{3}
N. Hayashi, \textit{Local existence in time of small solutions to the Davey-Stewartson systems}, Ann.
Inst. H. Poincar\'e, {\bf 65} (1996), 313--366.
\bibitem{4}
 N. Hayashi, J. C. Saut, \textit{Global existence of small solutions to the Davey-Stewartson and the
Ishimori systems}, Differential and Integral Equations {\bf 8} (1995), no. 7, 1657--1675.

\bibitem{5}
 N. Hayashi, H. Hirata, \textit{Local existence in time of small solutions to the elliptic-hyperbolic Davey-Stewartson system in the usual Sobolev space}, Proceedings of the Edinburgh Mathematical Society II {\bf 40} (1997), no. 3, 563--581.
\bibitem{Kenig}
C. E. Kenig, G. Ponce, L. Vega, \textit{Quadratic forms for the 1-D semilinear
Schr\"odinger equation}, Trans. Amer. Math. Soc. {\bf348} (1996), 3323--3353.
\bibitem{Linares}
 F. Linares, G. Ponce, \textit{On the Davey-Stewartson systems}, Ann. Inst. H. Poincar\'e, Anal. Non Lin\'eaire {\bf 10} (1993), 523--548.
\bibitem{Guo}
C. X. Shen, B. L. Guo, \textit{Almost conservation law and global rough solutions to a nonlinear Davey-
Stewartson equation}, J. Math. Anal. Appl. {\bf318} (2006), no. 1, 365--379.

\bibitem{Temam}
R. Temam, \textit{Infinite-dimensional dynamical systems in mechanics and physics}, Applied Mathematical Sciences {\bf68}, Springer, 1997.

\bibitem{6}
M. Tsutsumi, \textit{Decay of weak solutions to the Davey-Stewartson systems}, J. Math. Anal. Appl.
{\bf 182} (1994), no. 3, 680--704.
\bibitem{Tao}
T. Tao, \textit{Multilinear weighted convolution of $L^2$ functions, and applications to nonlinear dispersive equation}, Amer. J. Math. {\bf123} (2001) 839--908.
\bibitem{Wang}
B. Wang, B. Guo, \textit{Attractors for the Davey-Stewartson systems on $\mathbb{R}^2$}, J. Math. Phys. {\bf 38} (1997), 2524--2534.

\bibitem{Yang}
H. Yang, X. Fan, S. Zhu, \textit{Global Analysis for Rough Solutions to
the Davey-Stewartson System,} Hindawi Publishing Corporation
Abstract and Applied Analysis (2012),
vol. 2012, 1--22.
\end{thebibliography}
\end{document}